\documentclass[11pt]{article}
\usepackage{amsmath,amssymb,amsthm}
\usepackage{geometry}
\usepackage{mathtools}
\usepackage{mathrsfs}
\usepackage[hidelinks]{hyperref}
\usepackage{enumitem} 
\numberwithin{equation}{section}  
\usepackage{cite} 

\newtheorem{theorem}{Theorem}[section]
\newtheorem{lemma}{Lemma}[section]

\newtheorem{proposition}{Proposition}[section]
\newtheorem{remark}{Remark}[section]
\newtheorem{definition}{Definition}[section]

\title{Renormalized Solutions for a Class of Nonlinear Parabolic Equation with a 
	Lower Order Term and Variable Exponents}
\author{
	LI Chunjin \\
	School of Mathematics and Statistics, Hainan University, Haikou, China \\
	\and
	LI Shijun\thanks{Email: sjlee@hainanu.edu.cn} \\
	School of Mathematics and Statistics, Hainan University, Haikou, China \\
	\and
	XU Shaopeng\thanks{Corresponding author: xuxsp@126.com} \\
	School of Mathematics and Statistics, Hainan University, Haikou, China
}
\date{\notag}
\begin{document}
	\maketitle
	\noindent \textbf{Abstract:} We consider a class of nonlinear parabolic equations 
	
	\[
	\dfrac{\partial}{\partial t} b(u)-\nabla \cdot (A(x,t,u,\nabla u))+H(x,t,\nabla u)=f ,
	\]
	where $H$ is a nonlinear lower order term satisfied the Carath$\acute{e}$odory condition and
	\[
	\left\lvert H(x,t,\nabla u)\right\rvert\leqslant g(x,t)\left\lvert \nabla u\right\rvert^{\delta(x)} 
	\]
	with 
	\[
	\delta (x)=\frac{p(x)(N+1)-N}{(N+2)(p(x)-1)}(p^--1) \quad \text{and} \quad p^-=\underset{x\in\bar{\Omega}}{min}\,p(x).  
	\]
	By virtue of truncation metheod,the monotone operator theory
	and a gradient estimate we prove existence of renormalized solutions without coercivity condition on lower order term in the framework of variable exponents.
	\bigskip
	
	\noindent \textbf{Keywords:}  Parabolic equations; renormalized solutions; variable exponents; lower order term
	
	\vspace{1cm}

\section{Introduction}
	
In this paper, we study the class of nonlinear parabolic equations of the type
	
\begin{equation}\label{beginequation}
	\begin{cases} 
		\dfrac{\partial}{\partial t} b(u) - \nabla \cdot(A(x, t, u, \nabla u)) + H(x, t, \nabla u) = f & \text{in } Q_T \\ 
		u = 0 & \text{on } \partial\Omega \times (0, T) \\ 
		b(u(x, 0)) = b(u_0(x)) & \text{in } \Omega.
	\end{cases}
\end{equation}
where $\Omega$ is a bounded open subset of $R^N(N>1)$ with a Lipschitz boundary $\partial \Omega$. And we define a continuous real-valued function $p$:
	\[
	p:\bar{\Omega}\rightarrow [1,+\infty)
	\]
Let
	\[
	p^-=\min_{x\in\bar{\Omega}}p(x)
	\]
and 
	\[
	p^+=\max_{x\in\bar{\Omega}}p(x)
	\]
restricted by $1<p^-\leqslant p^+<+\infty$ naturally. The operator $-\nabla \cdot(A(x,t,u,\nabla u))$, which works on \(t\), defined in \eqref{beginequation}, is a Leray-Lions operator. That means the operator is coercive and grows like $|\nabla u|^{p(x)-1}$ with respect to $\nabla u$. The cylinder $Q_T=\Omega \times (0,T)$ is defined in the usual way for $T>0$, and the function $H$ is a \textbf{ Carathéodory function}. Finally, $f\in L^1(Q_T)$ and $b(u_0)\in L^1(\Omega)$.

The existence of a renormalized solution to equation \eqref{beginequation} intrigues us. But the difficulties of this problem are due to the lack of coercivity and due to the $L^1$ data.
To overcome the difficulties, R.J.Dipenna and P-L.Lions \cite{DiPerna1989} introduced renormalized solutions in the study of Boltzmann equation, and notion of renormalized solutions is extended to more general problems of parabolic \cite{Bendahmane2010} \cite{Blanchard1998}.

Many important results of equations with variable exponents were found in the last two decades \cite{Zhikov1987} \cite{Zhikov1992} \cite{Fan1996}. For example,by using $p(x)$-growth condition, the weak solutions of elliptic equations are Hölder continuous and its gradient have high integrability proved by Fan and Chao \cite{Fan1996}. In particular, the Lebesgue space with variable exponents was developed by Zhikov indicating that smooth truncation are not dense in $L^{p(x)}$ and $W^{k,p(x)}$ \cite{Zhikov2006}.

If $H\equiv 0$,$b(u)=u$ and $f\in L^1$, the existence of weak solutions was proved by Lucio Boccardo \cite{Boccardo1997} for $p>2-\frac{1}{N+1}$. Furthermore, the existence of renormalized solution and entropy solution for a nonlinear parabolic function with variable exponent was proved by Chao Zhang and Shulin Zhou \cite{Zhang2010}. For special case $ b(u)=u$, $H\equiv 0$ and which another lower order term is $\varPhi$, problem \eqref{beginequation} was studied by Rosaris Di Nardo \cite{DiNardo2010}(also see \cite{Blanchard2001}).M.Ben proved that nonlinear elliptic problem in performed domains exist at least a renormalied solution.

In the frame of variable exponents, Hamdaoui and Zhongqing Li obtained existence results of renormalied solutions in different problems \cite{ElHamdaoui2020} \cite{Li2022}. Moreover, if the right hand side is a bounded measure $\mu$ and p is a constant,the existence of a renormalized solution for \eqref{beginequation} was introduced \cite{Bouajaja2021}. Finally, the existence and uniquess of renormalized solution for nonlinear parabolic problem with two lower order terms was obtained,where $b(u)$ is equal to $u$ and p is a constant \cite{Abdellaoui2021} \cite{DiNardo2013} \cite{DiNardo2011}.

The main inspiration for this paper comes from \cite{DiNardo2011} and the main technique come from \cite{Blanchard2001} and \cite{Abdellaoui2021}. It is our purpose in this paper to prove an existence result of renormalied solution of \eqref{beginequation}. The paper is organized as follows.In section 2, some important definitions and properties are given and an important lemma is introduced to deal with lower order term $H$. In section 3, the definition of a renormalized solution is given and main result is proved in this paper which is the existence of a renormalized solution.

\section{Assumptions and Definitions of a Renormlized Solutions}

Throughout the paper, we assume that the following assumptions hold.
\begin{enumerate}
	\item $p(x)$ is a continuous function on $\bar{\Omega}$,
	$p^+=\max_{x\in\bar{\Omega}} p(x)$ and $p^-=\min_{x\in\bar{\Omega}}p(x)$. Here is a restriction $1<p^- \leq p(x) \leq p^+ < +\infty$. Moreover, there exist a constant $C>0$ such that 
	\begin{equation}
	|p(x) - p(y)|\leq -\dfrac{C}{\log|x-y|} \,\, \text{for every} \,\, x,y \in \Omega \,\, \text{satisfying} \,\, |x-y|\leq \frac{1}{2},      
	\end{equation}
	which is called the log-Hölder continuity condition.
	\item $b: R \rightarrow R$ is a strictly increasing $C^1$-function with $b(0)=0$, and there exist $b_0>0$ and $b_1>0$ such that
	\begin{equation}\label{as2}
		b_0\leq b'(s) \leq b_1, \,\, \text{for all} \,\, s\in R.
	\end{equation}
	\item $A(x,t,s,\xi) :Q_T\times R\times R^N \rightarrow R^N$ is a Carathéodory function satisfies the following conditions
	\begin{equation}\label{as3.1}
		A(x,t,s,\xi)\cdot\xi \geq \alpha \vert \xi \vert ^{p(x)},   
	\end{equation}
	\begin{equation}\label{as3.2}
		[A(x,t,s,\xi)-A(x,t,s,\xi')][\xi-\xi']>0, 
	\end{equation}
	\begin{equation}\label{as3.3}
		|A(x,t,s,\xi)| \leq C(L(x,t)+|\xi|^{p(x)-1}),
	\end{equation}
	where $C>0$ and $L(x,t)\in L^{p'(x)}(Q_T)$.
	\item $f(x,t) \in L^1(Q_T),u_0(x) \in L^1({\Omega})$.\\
	\item $H(x,t,\xi):Q_T\times R^N\rightarrow R $ satisfies the Carathéodory condition and
	\begin{equation}\label{as5.1}
		|H(x,t,\xi)| \leq g(x,t) |\xi|^{\delta(x)}
	\end{equation}
	with $g(x,t) \in L^{N+2,1}(Q_T)$ and 
	\begin{equation}\label{as5.2}
		\delta (x)=\dfrac{(N+1)p(x)-N}{(N+2)(p(x)-1)}(p^--1).
	\end{equation}
\end{enumerate}

Firstly, the Lorentz space $L^{r,1}(Q_T)$ and $L^{q,\infty}(Q_T)$ are explained in \cite{Baroni2013}. If $f^*$ denotes the decreasing rearrangenment of a measurable functions $f$,
\begin{equation}
	f^*(r)= \inf\{s\geq0:meas{(x,t)\in Q_T:|f(x,t)|>0 }<r\},r\in[0,|Q_T|].
	\nonumber
\end{equation}
$L^{q,1}(Q_T)$ and $L^{q,\infty}(Q_T)$ are the spaces of Lebesgue measurable functions such that 
\[
\|f\|_{L^{q,1}(Q_T)}=\left(\int_0^{|Q_T|}{f^*(r)r^{\frac{1}{q}}\frac{\mathrm{d}r}{r}}\right)<+\infty,
\]
\[
\| f \| _{L^{q,\infty}(Q_T)}=\sup_{r>0}r[meas\{(x,t)\in Q_T:|f(x,t)|>r\}]^{\frac{1}{q}}<+\infty.
\]
For $1<q<+\infty $, we have the Hölder inequality
\begin{equation*}
	\int_{Q_T}{ \left\lvert fg \right\rvert }\leqslant \left\lVert f \right\rVert _{L^{q,\infty}(Q_T)} |g|_{L^{q',1}(Q_T)}, \,\, \forall f \in L^{q,\infty}(Q_T), \forall g\in L^{q',1}(Q_T)
\end{equation*}
Thoughout this paper,the variable exponent Lebesgue space $L^{p(x)}(\Omega)$ is defined as
\begin{equation*}
	L^{p(x)}(\Omega)=\left\{u: u \text{ is a measurable function such that} \int_{\Omega}{|u|}^{p(x)}\mathrm{d}x < +\infty 
	\right\}
\end{equation*}
endowed with the norm
\begin{equation*}
	\|u\|_{L^{p(x)}(\Omega)}=\inf\left\{\lambda >0:\int_{\Omega}{\left|\frac{u}{\lambda} \right|}^{p(x)}\mathrm{d}x \leq 1 \right\}.
\end{equation*}
As a matter of fact
\begin{equation*}
	\min \left\{\|u\|_{L^{p(x)}\left(\Omega\right)}^{p^-},\|u\|_{L^{p(x)}(\Omega)}^{p^+} \right\} \leq \int_{\Omega}{\left|u\right|^{p(x)}}\mathrm{d}x \leq \max\left\{\|u\|_{L^{p(x)}(\Omega)}^{p^-},\|u\|_{L^{p(x)}(\Omega)}^{p^+} \right\}.
\end{equation*}
The variable exponent Sobolev space $W^{k,p(x)}(\Omega)$ is defined as
$$ W^{k,p(x)}(\Omega)=\left\{u \in L^{p(x)}(\Omega):D^{\beta} u \in L^{p(x)}(\Omega),|\beta|\leq k \right\} $$
endowed with the norm
$$ \|u\|_{W^{k,p(x)}(\Omega)}=\sum_{\left|\beta\right|\leq k}{\left| D^{\beta}u \right|_{L^{p(x)}(\Omega)}}.$$
Recalling a functional space
$$ V=\left\{v \in L^{p^-}\left([(0,T); W^{1,p(x)}_{0} (\Omega)\right):|\nabla u|\in L^{p(x)}(Q_T) \right\}, $$
which,endowed with the norm
$$
\|v\|_V \coloneqq |\nabla v\| _{L^{p(x)}(Q_T)},
$$
is a separable and reflexive Banach space.
\begin{proposition} 
	~\\
	\begin{enumerate}
		\item For any function $u\in L^{p(x)}(\Omega)$ and $ v\in L^{p'(x)}(\Omega)$,we have
		\begin{equation}
			\left\lvert \int_{\Omega}{uv}\mathrm{d}x\mathrm{d}t \right\rvert\leqslant(\frac{1}{p^-}+\frac{1}{p'^-})\left\lVert u\right\rVert _{L^{p(x)}(\Omega)}\left\lVert v\right\rVert _{L^{p'(x)}(\Omega)}\leqslant2 \left\lVert u\right\rVert _{L^{p(x)}(\Omega)}\left\lVert v\right\rVert _{L^{p'(x)}(\Omega)}.
			\nonumber
		\end{equation}
		\item	Let $p_1(x),p_2(x)$ be continuous functions with $ 1<p_1(x) \leqslant p_2(x)<+\infty$ for any $x\in {\bar{\varOmega}}$, then
		$$L^{p_2(x)}(\Omega)\hookrightarrow L^{p_1(x)}(\Omega).$$
		\item	Let a variable exponent $p(x)$ satisfy the log-Hölder continuity such that 
		$$1<p^-\leqslant p(x) \leqslant p^+<N.$$
		We have
		$$ \forall u\in W_0^{1,p(x)}(\Omega):\left\lVert u \right\rVert _{L^{p^{\ast }(x)}(\Omega)}\leqslant \left\lvert \nabla u\right\rVert _{L^{p(x)}(\Omega)},$$
		where $C=C(N,p^-,p^+)$ and
		$$\frac{1}{p^*(x)}=\frac{1}{N}-\frac{1}{p(x)}$$
		for $p(x)<N$ a.e $x\in \bar{\Omega}$.
	\end{enumerate}
\end{proposition}
\begin{remark}
	Notice that $V\cap L^\infty (Q_T)$, endowed with the norm
	$$
	\left\| v \right\| _{V\cap L^{\infty}\left( Q_T \right)} \coloneqq \max \left\{ \left\| v \right\| _V,\left\| v \right\| _{L^{\infty}\left( Q_T \right)} \right\} .
	$$
	In fact, it is the dual sapce of the Banach space $V^{\ast} +L^1(Q_T)$, endowed with the norm
	$$
	\left\| v \right\| _{V^*+L^1\left( Q_T \right)}:=inf\,\,\left\{ \left\| v_1 \right\| _{V^*}+\left\| v_2 \right\| _{L^1\left( Q_T \right)}:v=v_1+v_2,v_1=V^*,v_2=L^1\left( Q_T \right) \right\} 
	,$$
	where V* is dual of V.
\end{remark}
\begin{proof}
	\cite{ElHamdaoui2020}
\end{proof}
\begin{definition}
	A measurable function $u$ is a renormalized solution to \eqref{beginequation} if
	\begin{equation}
		b(u)\in L^\infty\left([0,T]; L^1(\Omega)\right), 
	\end{equation}
	\begin{equation}\label{deftku}
		T_k(u)\in L^{p^-}\left((0,T) ;W^{1,p(x)}_{0}(\Omega)\right) \text{ for any } k>0,
	\end{equation}
	\begin{equation}
		\lim_{n \rightarrow +\infty}\frac{1}{n}\iint_{\{(x,t)\in Q_T:|u|< n\}}{A(x,t,u,\nabla u)}\cdot\nabla u \, \mathrm{d}x\mathrm{d}t=0
	\end{equation}
	and if for every function S in $W^{2,\infty}(R)$ which is piecewise $C^1$ and such that
	$S'$ has a compact support, there holds
	\begin{equation}\label{def1}
		\begin{split}
			\dfrac{\partial \Theta _S\left( u \right)}{\partial t}-\nabla \cdot \left( S'\left( u \right) A\left( x,t,u,\nabla u \right) \right) +S''\left( u \right) A\left( x,t,u,\nabla u \right) \cdot \nabla u
			\\+H\left( x,t,\nabla u \right) S'\left( u \right) =fS'\left( u \right) \,\,\text{in}\,\,D'\left( Q_T \right) 
		\end{split}
	\end{equation}
	and
	\begin{equation}\label{def1.1}
		\Theta_S(u)(t=0)=\Theta_S(u_0),\quad a.e. \text{in }\Omega,
	\end{equation}
	where $\Theta _S(z)=\int_0^z{b'(s)S'(s)}\,\mathrm{d}s$.
\end{definition}
\begin{remark}
	Equation \eqref{def1} is formally obtained through multiplication of \eqref{beginequation} by $S'(u)$. However, while
	$A(x,t,u,\nabla u) \text{ and } H(x,t,\nabla u)$ does not in general make sense in $D'(Q_T)$, all the terms in \eqref{def1} have a
	meaning in $D'(Q_T).$
\end{remark}
We now claim that
\begin{equation}
	\dfrac{\partial \Theta _S(u)}{\partial t}\in V^{\ast}+L^1(Q_T)
\end{equation}
and
\begin{equation}
	\Theta _S(u)\in V.   
\end{equation}
Indeed, if we assume $M$ is a constant such that $ \operatorname{supp} S'\subset [-M,M]$, we have
\begin{enumerate}
	\item $\left\lvert \Theta _S(u)\right\rvert=\left\lvert \Theta_S(T_M(u))\right\rvert\leqslant b_1M \left\lVert S' \right\rVert _{L^{\infty}(R)}\,\text{ belong to }\,L^{\infty}(Q_T)$.  
	\item Since $S'(u)A(x,t,u,\nabla u)=S'(u)A(x,t,T_M(u),\nabla T_M(u))\,\,a.e\,\text{in }Q_T$,we otain from \eqref{as3.3} and \eqref{deftku} that
	$$S'(u)A(x,t,T_M(u),\nabla T_M(u))\in \left(L^{p'(x)}(Q_T)\right)^N.$$ 
	\item Since $S''(u)A(x,t,u,\nabla u)\nabla u=S''(u)A(x,t,T_M(u),\nabla T_M(u))\nabla T_M(u)$,
	we have $$S''(u)A(x,t,T_M(u),\nabla T_M(u))\nabla T_M(u)\in L^1(Q_T).$$ 
	\item Since $S'(u)H(x,t,\nabla u)=S'(u)H(x,t,\nabla T_M(u))\,\,a.e\,\text{ in }\,Q_T$, it follows from \eqref{as5.1} and \eqref{deftku} that
	\begin{equation}
		\begin{split}
			&\iint_{Q_T}{\left\lvert S'(u)H(x,t,\nabla T_M(u))\right\rvert}dxdt \\ 
			\leqslant & C \left\lVert S' \right\rVert _{L^{\infty}(R)}\left\lVert g\right\rVert _{L^{\frac{p(x)(N+2)}{N+p(x)}}(Q_T)}\left\lVert \left\lvert \nabla T_M (u)\right\rvert ^{\delta (x)} \right\rVert _{L^{\frac{(N+2)p(x)}{N(p(x)-1)+p(x)}}}\\ 
			\leqslant & C\left\lVert S'\right\rVert _{L^{\infty}(R)}\left\lVert g \right\rVert _{L^{N+2}(Q_T)}\left\lVert \nabla T_M(u)\right\rVert^{\eta }_{L^{p(x)}(Q_T)}\leqslant C,  
		\end{split}
		\nonumber
	\end{equation}
	where
	\begin{equation}
		\eta=
		\begin{cases}
			\dfrac{p^-}{l^-},&if\,\left\lVert \nabla T_M(u)\right\rVert _{L^{p'(x)(p^--1)}(Q_T)}\geqslant 1 \\
			\dfrac{(p^+)'(p^--1)}{l^+},&if \,\left\lVert \nabla T_M(u)\right\rVert _{L^{p'(x)(p^--1)}(Q_T)}\leqslant 1
		\end{cases},
		l(x)=\frac{(N+2)p(x)}{N(p(x)-1)+p(x)}.
		\nonumber
	\end{equation}
	Then we have $S'(u)H(x,t,\nabla T_M(u))\in L^1(Q_T)$.
\end{enumerate}
The above considerations show that \eqref{def1} holds in $D'(Q_T)$ and that
$$\frac{\partial \Theta _S(u)}{\partial t}\in V^{\ast}+L^1(Q_T).$$
The assumption of $\Theta _S(u)$ implies that
$$\left\lvert \nabla \Theta _S(u)\right\rvert \leqslant b_1\left\lVert S'\right\rVert _{L^{\infty}(R)}\left\lvert \nabla T_M(u)\right\rvert  $$
and
$$\Theta _S(u)\in L^{p^-}(0,T;W_0^{1,p(x)}(\Omega)).$$
\noindent Then we deduce that $\Theta _S(u)\in C((0,T);L^1(\Omega))$, so that the initial condition \eqref{def1.1} makes sense.
\begin{definition}
	We denote a class of measurable functions $p:R^N\rightarrow ( 0,\infty]$ such that $0<p^-\leqslant p(x)\leqslant  p^+< +\infty$ by $\mathcal{P}$ . Let $ p(x)\in  \mathcal{P}(R^N) $ 
	and let $0<q\leqslant +\infty$. Then $L^{p(x),q}(R^N)$ is the collection of all measurable functions $f:R^N\rightarrow \mathbb{C} $ such that
	\begin{equation}
		\left\| f \right\| _{L^{p\left( x \right) ,q}\left( R^N \right)}:=
		\begin{cases}
			\left( \displaystyle \int_0^{+\infty}{\lambda ^q\left\| \chi _{\left\{ x\in R^N:\left| f \right|>\lambda \right\}} \right\| _{L^{p\left( x \right)}\left( R^N \right)}^{q}\frac{\mathrm{d}\lambda}{\lambda}} \right) ^{\frac{1}{q}},&\text{if}\,\,0<q<+\infty\\
			\,\,\displaystyle \sup _{\lambda >0}\lambda \left\| \chi _{\left\{ x\in R^N:\left| f\left( x \right) \right|>\lambda \right\}} \right\| _{L^{p\left( x \right)}\left( R^N \right)},&\text{if}\,\,q=\infty
		\end{cases} 
		\nonumber           
	\end{equation}
	is finite\text{\cite{Jiao2019}}.
\end{definition}
\begin{lemma}\label{lemma1}\cite{Bendahmane2010}
	Let $V^{\ast}$ be dual space of V. Then we have the following continuous dense embedding:
	\begin{equation}
		L^{p^+}(0,T;W_0^{1,p(x)}(\Omega))\hookrightarrow  V \hookrightarrow L^{p^-}(0,T;W_0^{1,p(x)}(\Omega)),
	\end{equation}
	in particular,since $D'(Q_T)$ is dense in $L^{p^+}(0,T;W_0^{1,p(x)}(\Omega))$,it is dense in V and
	for the corresponding dual space.We have
	\begin{equation}
		L^{(p^-)'}(0,T;W_0^{-1,p'(x)}(\Omega))\hookrightarrow  V^{\ast}\hookrightarrow L^{(p^+)'}(0,T;W_0^{-1,p'(x)}(\Omega)).
	\end{equation}
\end{lemma}
\begin{lemma}\label{lemma2}\cite{Kempka2014}
	Let $p_1,p_2 $ be  continuous functions in $\bar{\Omega}$  with $1<p_1(x)\leqslant p_2(x)<+\infty$ for any $x\in \bar{\varOmega}$. Then 
	$$L^{p_2(x),\infty}(\Omega)\hookrightarrow L^{p_1(x),\infty}(\Omega).$$
\end{lemma}
\begin{lemma}\label{lemma3}
	Let $p(x)$ be a continuous function with $1<p^-\leqslant p(x)\leqslant p^+<+\infty$ and $s\in(0,+\infty)$.If $f\in L^{sp(x),\infty}(\Omega)$,then 
	$\left\lvert f\right\rvert^{s}\in L^{p(x),\infty}(\Omega) $ with the estimate
	\begin{equation}\label{lemmaexhnge}
		\left\| f \right\| _{L^{sp\left( x \right) ,\infty}\left( \Omega \right)}^{s}=\left\| \left| f \right|^s \right\| _{L^{p\left( x \right) ,\infty}\left( \Omega \right)}.
	\end{equation}
\end{lemma}
\begin{proof}
	By definition \eqref{as2}, we have
	\begin{equation}
		\begin{split}
			&\left\lVert f\right\rVert _{L^{sp(x),\infty}(\Omega)}
			\\= & \sup_{\lambda >0}\,\lambda \left\lVert \chi _{\left\{x \in \Omega:\left\lvert f \right\rvert > \lambda \right\}} \right\rVert _{L^{sp(x)}(\Omega)}
			\\= & \sup_{\lambda >0}\,\lambda \,\left\{\inf_{\mu >0}\left\{\mu:\int_{\Omega}\left\lvert \frac{\chi_{\left\{x\in \Omega:\left\lvert f(x)\right\rvert>\lambda  \right\} }}{\mu}\right\rvert^{sp(x)}dx\leqslant 1 \right\} \right\}                
			\\= & \sup_{\lambda >0}\,\lambda \,\left\{\inf_{\mu >0}\left\{\mu:\int_{\Omega}\left\lvert \frac{\chi_{\left\{x\in \Omega:\left\lvert f(x)\right\rvert>\lambda  \right\} }}{\mu^s}\right\rvert^{p(x)}dx\leqslant 1 \right\} \right\}.                
		\end{split} 
		\nonumber
	\end{equation}
	Taking $\lambda =\tau ^{\frac{1}{s}}$ and $\mu=h^{\frac{1}{s}}$, the above equality can be read as
	\begin{equation}
		\begin{split}
			&\left\lVert f\right\rVert _{L^{sp(x),\infty}(\Omega)} \\
			=&\sup_{\tau >0}\,\tau^{\frac{1}{s}}\left\{\sup_{h>0}\left\{h^{\frac{1}{s}}:\int_{\Omega} \left\lvert 
			\frac{ \chi _{\left\{\left\lvert f(x)\right\rvert^s > \tau\right\} }}{h}\right\rvert^{p(x)}dx \leqslant 1 \right\} \right\} \\
			=&\left[\sup_{\tau >0}\,\tau\left\{\sup_{h>0}\left\{h:\int_{\Omega} \left\lvert 
			\frac{ \chi _{\left\{\left\lvert f(x)\right\rvert^s > \tau\right\} }}{h}\right\rvert^{p(x)}dx \leqslant 1 \right\} \right\}\right]^{\frac{1}{s}} \\
			=&\left\lVert \left\lvert f\right\rvert ^s\right\rVert ^{\frac{1}{s}}_{L^{p(x),\infty}(\Omega)}. 
		\end{split}
		\nonumber
	\end{equation}
	It follows that $\left\lvert f\right\rvert ^s\in L^{p(x),\infty}(\Omega)$ and \eqref{lemmaexhnge} holds true.
\end{proof}
\noindent Similarly, we have following lemma.
\begin{lemma}\label{lemma4}
	Let $p(x)$ be a continuous function satisfies $1<p^-\leqslant p(x)\leqslant p^+<+\infty$ and $s\in(0,+\infty)$. If $\left\lvert f\right\rvert^s \in L^{p(x),\infty}(\Omega)$, then 
	$\left\lvert f\right\rvert\in L^{sp(x),\infty}(\Omega) $ with the estimate
	\begin{equation}
		\left\| \left\lvert f \right\rvert^s \right\| _{L^{p\left( x \right) ,\infty}\left( \Omega \right)}=\left\| \left| f \right| \right\|^s _{L^{sp\left( x \right) ,\infty}\left( \Omega \right)}.
		\label{(2.15')}
	\end{equation}
\end{lemma}
\begin{lemma}\label{lemma5}
	Let $s(x), p(x) $ be two continuous functions in $\bar{\Omega}$ with $1<s(x)p(x)<\infty$ for any $x \in \bar{\Omega}$. Then for all $f\in L^{s^+p(x),\infty}(Q_T)$, it holds true that
	\begin{equation}
		\left\| \left| f \right|^{s\left( x \right)} \right\| _{L^{p\left( x \right) ,\infty}\left( Q_T \right)}\leqslant C\left( \left\| f \right\| _{L^{s^+p\left( x \right) ,\infty}\left( Q_T \right)}^{s^+}+1\right), 
		\label{(2.16)}
	\end{equation}
	where $C$ is a positive constant.
	\begin{proof}
		Set $\chi $ is a characteristic function in $Q_T$. Then
		\begin{equation}
			\begin{split}
				&\left\lVert \left\lvert f\right\rvert^{s(x)} \right\rVert _{L^{p(x),\infty}(Q_T)} \\
				=&\left\lVert \chi_{\left\{(x,t)\in Q_T:\left\lvert f\right\rvert >1 \right\} }\left\lvert f\right\rvert^{s(x)}+
				\chi_{\left\{(x,t)\in Q_T:\left\lvert f\right\rvert \leqslant 1\right\}  }\left\lvert f\right\rvert^{s(x)}\right\rVert _{L^{p(x),\infty}(Q_T)}\\
				\leqslant& C\left\lVert \chi_{\left\{(x,t)\in Q_T:\left\lvert f\right\rvert >1 \right\} }\left\lvert f\right\rvert^{s(x)}\right\rVert _{L^{p(x),\infty}(Q_T)} 
				+C \left\lVert \chi_{\left\{(x,t)\in Q_T:\left\lvert f\right\rvert \leqslant 1\right\}  }\left\lvert f\right\rvert^{s(x)}\right\rVert _{L^{p(x),\infty}(Q_T)}\\
				=&C(A_1+A_2).\nonumber
			\end{split}
		\end{equation}
		Firstly we handle $A_1$,and it follows from lemma \eqref{lemma3} that
		\begin{equation}
			\begin{split}
				A_1\leqslant\left\| \left\lvert f \right\rvert^{s^+} \right\| _{L^{p(x),\infty}(Q_T)}=\left\| f\right\|^{s^+}_{L^{s^+p(x),\infty}(Q_T).\nonumber} 
			\end{split}
		\end{equation}
		Similarly ,using lemma\eqref{lemma2}and lemma\eqref{lemma3},we have 
		\begin{equation}
			\begin{split}
				A_2 &\leqslant \left\lVert \left\lvert f\right\rvert^{s^-} \right\rVert _{L^{p(x),\infty}(Q_T)}
				\\&= \left\lVert  f \right\rVert ^{s^-}_{L^{s^-p(x),\infty}(Q_T)}
				\\&\leqslant\left(C\left\lVert f\right\rVert _{L^{s^+p(x),\infty}(Q_T)}+1\right)^{s^-}
				\\&\leqslant C\left(\left\lVert f\right\rVert _{L^{s^+p(x),\infty}(Q_T)}+1\right)^{s^+}
				\\&\leqslant C\left(\left\lVert f\right\rVert^{s^+} _{L^{s^+p(x),\infty}(Q_T)}+1\right).\nonumber
			\end{split}
		\end{equation} 
		Therefore,we get 
		\begin{equation}
			\left\lVert \left\lvert f\right\rvert^{s(x)} \right\rVert _{L^{p(x),\infty}(Q_T)} \leqslant C\left(\left\lVert f\right\rVert^{s^+} _{L^{s^+p(x),\infty}(Q_T)}+1\right),\nonumber
		\end{equation}
		where C is a positive constant.
	\end{proof}
\end{lemma}
In order to deal with lower order term $H$, We provide two estimates in Lorentz spaces.
\begin{lemma}\label{lemma6}\cite{DiNardo2011} Assume that $Q_T=\Omega \times (0,T)$ with $\Omega $ open subset of $R^N$ of finite measure and $p>1$. Let be $u$ a measurable function satisfying
	$$T_k(u)\in L^{\infty}(0,T;L^2(\Omega))\cap L^p(0,T;W_0^{1,p}(\Omega))\,\text{ for every }\,k>0$$
and such that
	\begin{equation}
		\underset{t\in(0,T)}{sup}\int_{\Omega}{ \left\lvert T_k(u(t))\right\rvert }^2+\iint_{Q_T}\left\lvert \nabla T_k(u)\right\rvert ^p\leqslant kM+L, 
	\end{equation} 
	where $M$ and $L$ are constants.
	Then $$u^{\frac{N(p-1)+p}{N+p}}\in L^{\frac{N+p}{N},\infty}(Q_T)\,\text{ and }\,\left\lvert \nabla u\right\rvert ^{\frac{N(p-1)+p}{N+2}}\in L^{\frac{N+2}{N+1},\infty}(Q_T)$$ 
hold. And we also get
	\begin{equation}
		\left\lVert \left\lvert u\right\rvert^{\frac{N(p-1)+p}{N+p}} \right\rVert _{L^{\frac{N+p}{p},\infty}(Q_T)}
		\leqslant C[M+\left\lvert Q_T\right\rvert^{\frac{Np}{N+2}} L^{\frac{N(p-1)+p}{(N+2)p}}],
	\end{equation}
	\begin{equation}
		\left\lVert \left\lvert \nabla u\right\rvert ^{\frac{N(p-1)+p}{N+2}}\right\rVert _{L^{\frac{N+2}{N+1},\infty}(Q_T)}\leqslant C(N,p)[M+\left\lvert Q_T\right\rvert ^{\frac{N}{(N+2)p}}L^{\frac{N(p-1)+p}{(N+2)p}}],  
	\end{equation}
	where $C(N,p)$ is a constant depending only on $p$ and $N$.
\end{lemma}
\begin{lemma}\label{lemma7}
	Let $\Omega$ be open set of $R^N$ of finite measure and $u$ is a measurable function such that for every positive $K$
	\begin{equation}
		T_k(u)\in L^{p^-}((0,T);W_0^{1,p(x)}(\Omega))\cap L^{\infty} (0,T;L^2(\Omega))
	\end{equation}
\end{lemma}
\noindent and
\begin{equation}
	\sup_{t\in(0,T)} \int_{\Omega}{ \left\lvert T_k(u)\right\rvert^2  } +\iint_{Q_T}{\left\lvert \nabla T_k(u)\right\rvert ^{p(x)}}\leqslant CMk.  
\end{equation}
Then
\begin{equation}
	\left\lVert \left\lvert u\right\rvert^{\frac{N(p^--1)+p^-}{N+p^-}} \right\rVert _{L^{\frac{N+p^-}{N},\infty}(Q_T)}\leqslant C[M+\left\lvert Q_T\right\rvert^{\frac{N((p^-)^2-1)+(N+1)p^-}{N+2}} ], 
\end{equation}
\begin{equation}
	\left\lVert \left\lvert \nabla u\right\rvert^{\frac{N(p^--1)+p^-}{N+2}} \right\rVert _{L^{\frac{N+2}{N+1},\infty}(Q_T)}\leqslant C[M+\left\lvert Q_T\right\rvert^{\frac{N+1}{N+2}}],
	\label{(2.21)}
\end{equation}
where $C=C(N,p^-)$ is a constant.
\begin{proof}
	Using generalized Young's inequality,we have
	\begin{equation}
		\iint_{Q_T}{\left\lvert \nabla T_k(u)\right\rvert ^{p^-}}\mathrm{d}x\mathrm{d}t \leqslant \iint_{Q_T}{\left\lvert \nabla T_k(u)\right\rvert ^{p(x)}}\mathrm{d}x\mathrm{d}t+\left\lvert Q_T\right\rvert.
		\nonumber
	\end{equation}
	Then (\ref{(2.21)}) can be showed that 
	\begin{equation}
		\begin{split}
			&\sup_{t\in(0,T)}\int_{\Omega}{\left\lvert T_k(u)\right\rvert^2 }dx+\iint_{Q_T}{\left\lvert \nabla T_k(u)\right\rvert^{p^-} }\mathrm{d}x\mathrm{d}t
			\\ \leqslant& \underset{t\in (0,T)}{sup}\int_{\Omega}{\left\lvert T_k(u)\right\rvert^2 }+\iint_{Q_T}{\left\lvert \nabla T_k(u)\right\rvert^{p(x)} }\mathrm{d}x\mathrm{d}t+\left\lvert Q_T\right\rvert 
			\\ \leqslant& CMk+\left\lvert Q_T\right\rvert. 
		\end{split}
		\nonumber
	\end{equation}
	From lemma 2.5 it follows that
	\begin{equation}
		\left\lVert \left\lvert u\right\rvert^{\frac{N(p^--1)+p^-}{N+p^-}} \right\rVert _{L^{\frac{N+p^-}{N},\infty}(Q_T)}\leqslant C[M+\left\lvert Q_T\right\rvert^{\frac{N((p^-)^2-1)+(N+1)p^-}{N+2}} ], 
		\nonumber
	\end{equation}
	\begin{equation}
		\left\lVert \left\lvert \nabla u\right\rvert^{\frac{N(p^--1)+p^-}{N+2}} \right\rVert _{L^{\frac{N+2}{N+1}},\infty(Q_T)} \leqslant C[M+\left\lvert Q_T\right\rvert^{\frac{N+1}{N+2}} \,],
		\nonumber
	\end{equation}
	where $C=C(N,p^-)$ is a constant.
\end{proof}

\section{Main Results}

The main results of the present paper are collected in the following theorem.
\begin{theorem}	\label{existence}
	Under the assumtions \eqref{as2}-\eqref{as5.1} there exists at least a renormalized solution to \eqref{beginequation}.
\end{theorem}
\begin{proof}
The above theorem is to be proved in five steps.
	
\textbf{Step 1.} Approximate equation

Let $ \varepsilon >0 $,consider the following approxiamate problem
	\begin{equation}\label{approximate}
		\begin{cases}
			\dfrac{\partial b_{\varepsilon}(u^{\varepsilon})}{\partial t}-\nabla\cdot\left( A_{\varepsilon}\left( x,t,u^{\varepsilon},\nabla u^{\varepsilon}\right) \right) +H_{\varepsilon}\left( x,t,\nabla u^{\varepsilon} \right) =f^{\varepsilon}\,\,  \,\,&\text{in }Q_T\\
			u_{\varepsilon}(x,t)=0&\text{on }\partial \Omega \times (0,T)\\
			b_{\varepsilon}(u^{\varepsilon})(t=0)=b_{\varepsilon}(u_0^{\varepsilon})&\text{in }\Omega,
		\end{cases} 
	\end{equation}
	where
	\begin{equation}\label{btr}
		b_{\varepsilon}(r)=T_{\frac{1}{\varepsilon}}(b(r))+\varepsilon r,\;\forall r\in R,
	\end{equation}
	\begin{equation}\label{aa}
		A_{\varepsilon}(x,t,s,\xi )=A(x,t,T_{\frac{1}{{\varepsilon}}}(s),\xi),
	\end{equation}
	\begin{equation}\label{ht}
		H_{\varepsilon}(x,t,\xi)=T_{\frac{1}{\varepsilon}}(H(x,t,\xi)).
	\end{equation}
And it is easy to see that
	\begin{equation}\label{fl}
		f^{\varepsilon}\in L^{p'(x)}(Q_T)\text{ and }u_0^{\varepsilon}\in\,D(\Omega),
	\end{equation}
	\begin{equation}\label{ffqt}
		\left\lVert f^{\varepsilon}\right\rVert _{L^1(Q_T)}\leqslant \left\lVert f\right\rVert _{L^1(Q_T)}\text{and}f^{\varepsilon} \to f\text{ strongly in }L^1(Q_T)\text{ and }a.e\text{ in }Q_T ,  
	\end{equation}
	\begin{equation}\label{bbqt}
		\left\lVert b_{\varepsilon}(u^{\varepsilon}_0)\right\rVert _{L^1(\Omega)}\leqslant \left\lVert b(u_0)\right\rVert _{L^1(\Omega)}  \text{ and } b_{\varepsilon}(u^{\varepsilon}_0 )\to b(u_0)\text{ strongly in }L^1(\Omega)\text{ and }a.e\text{ in }\Omega.
	\end{equation}
	By pseudo-monotone operator theory in \cite{Lions1969}, the existence of weak solutions $u^{\varepsilon}\in V$ to \eqref{approximate} can be proved.
	
	\textbf{Step2.} A priori esimate
	
	\indent Let $k>0$, making use of $T_k(u^{\varepsilon})\chi _{(0,t)}$ as a test function in \eqref{approximate} for almost any $t\in(0,T)$, it results
	\begin{equation}
		\begin{split}
			\int_{\Omega}{\varTheta^{\varepsilon} _k\left( u^{\varepsilon} \right) \left( t \right) \mathrm{d}x}+\iint_{Q_{t}}{A\left( x,t,T_{\frac{1}{\varepsilon}}\left( u^{\varepsilon} \right) ,\nabla u^{\varepsilon} \right) \cdot \nabla T_k\left( u^{\varepsilon} \right)}\mathrm{d}x\mathrm{d}s
			\\\leqslant k\iint_{Q_t}{g(x,t)\left\lvert \nabla u ^{\varepsilon}\right\rvert^{\delta(x)}}\mathrm{d}x\mathrm{d}s +\iint_{Q_{t}}{f^{\varepsilon}T_k\left( u^{\varepsilon} \right)}\mathrm{d}x\mathrm{d}t+\int_{\Omega}{\Theta_k^{\varepsilon}(u^{\varepsilon}_0)}\mathrm{d}x
		\end{split}
	\end{equation}
	where $\Theta_k^{\varepsilon}(r)=\int_0^r{T_k(s)b'_{\varepsilon}(s)}\mathrm{d}s$. Denoting by
	$$
	\bar{T}_k\left( s \right) =\int_0^s{T_k\left( \sigma \right) \mathrm{d}\sigma =\left\{ \begin{array}{c}
			\dfrac{s^2}{2},if\,\,\left| s \right|<k\\
			k\left| s \right|-\dfrac{k^2}{2}\, ,if\,\left| s \right|\geqslant k\\
		\end{array} \right.}.
	$$
	Throught the use of assumption of $b$, we have
	\begin{equation}\label{thetatk}
		\int_{\Omega}{\Theta _k^{\varepsilon}(u^{\varepsilon})}(t)\mathrm{d}x\geqslant \frac{b_0}{2}\int_{\Omega}{\left\lvert T_k(u^{\varepsilon})\right\rvert^2 }\mathrm{d}x
	\end{equation}
	and
	\begin{equation}\label{0kb}
		0\leqslant\int_{\Omega}{\Theta _k^{\varepsilon}(u_0^{\varepsilon})}\mathrm{d}x\leqslant k\int_{\Omega}{\left\lvert b_{\varepsilon}(u_0^{\varepsilon})\right\rvert }\mathrm{d}x\leqslant k \left\lVert b(u_0)\right\rVert _{L^1(\Omega)}.
	\end{equation}
	If we take the supremum for $t\in(0,t_1)$,where $t_1\in(0,T)$ will be choosen later, by \eqref{as2} \eqref{thetatk} and \eqref{0kb}, we have
	\begin{equation}
		\begin{split}
			\frac{b_0}{2} &\sup_{t\in(0,t_1)}\int_{\Omega}{\left\lvert T_k(u^{\varepsilon})\right\rvert^2 }\mathrm{d}x+\alpha \iint_{Q_{t_1}}{\left\lvert \nabla T_k(u^{\varepsilon})\right\rvert^{p(x)} }\mathrm{d}x\mathrm{d}t
			\\ &\leqslant k\left(\left\lVert \left\lvert \nabla u^{\varepsilon}\right\rvert^{\delta (x)} \right\rVert _{L^{\frac{N+2}{N+1},\infty}(Q_{t_1})}\left\lVert g\right\rVert _{L^{N+2,1}(Q_{t_1})}+\left\lVert f^{\varepsilon}\right\rVert _{L^1(Q_T)}+\left\lVert b(u_0)\right\rVert _{L^1(\Omega)} \right) 
			\\ &\leqslant Mk,    
		\end{split}
	\end{equation}
	where
	$$M=\left(\left\lVert \left\lvert \nabla u^{\varepsilon}\right\rvert^{\delta (x)} \right\rVert _{L^{\frac{N+2}{N+1},\infty}(Q_{t_1})}\left\lVert g\right\rVert _{L^{N+2,1}(Q_{t_1})}+\left\lVert f^{\varepsilon}\right\rVert _{L^1(Q_T)}+\left\lVert b(u_0)\right\rVert _{L^1(\Omega)} \right).$$
	Using lemma \eqref{lemma2}, lemma \eqref{lemma4} and lemma \eqref{lemma6}, we obtain
	\begin{equation}
		\begin{aligned}[b]
			&\left\lVert \left\lvert \nabla u^{\varepsilon}\right\rvert^{\delta (x)} \right\rVert _{L^{\frac{N+2}{N+1},\infty}(Q_{t_1})} \\
			=&\left\lVert \left\lvert \nabla u^{\varepsilon}\right\rvert^{(p^--1)\frac{N(p(x)-1)+p(x)}{(N+2)(p(x)-1)}} \right\rVert _{L^{\frac{N+2}{N+1},\infty}(Q_{t_1})} \\
			\leqslant& C_1 \left\lVert  \left\lvert \nabla u^{\varepsilon}\right\rvert ^{(p^--1)}\right\rVert ^{\frac{N(p^--1)+p^-}{\left(N+2\right)(p^--1) }} _{L^{\frac{N(p^--1)+p^-}{\left(N+1\right)(p^--1) },\infty}(Q_{t_1})}+C_1 \\
			=&C_1 \left\lVert \nabla u^{\varepsilon}\right\rVert _{L^{\frac{N(p^--1)+p^-}{N+1},\infty}(Q_{t_1})}^{\frac{N(p^--1)+p^-}{N+2}}+C_1 \\
			=&C_1\left\lVert \left\lvert \nabla u^{\varepsilon}\right\rvert ^{\frac{N(p^--1)+p^-}{N+2}}\right\rVert _{L^{\frac{N+2}{N+1},\infty}(Q_{t_1})}+C_1 \\
			\leqslant& C_2 M+C_3 \\
			\leqslant& C_4\left(\left\lVert \left\lvert \nabla u^{\varepsilon}\right\rvert^{\delta (x)} \right\rVert _{L^{\frac{N+2}{N+1},\infty}(Q_{t_1})}\left\lVert g\right\rVert _{L^{N+2,1}(Q_{t_1})}+\left\lVert f^{\varepsilon}\right\rVert _{L^1(Q_T)}+\left\lVert b(u_0)\right\rVert _{L^1(\Omega)} \right) +C_4, \\
		\end{aligned}
	\end{equation}
	where $C_4=C_4(\alpha ,b_0,p^+,p^-,N,\Omega)$.\\
	If we choose $t_1$ such that 
	\begin{equation}\label{c4gln}
		1-C_4\left\lVert g\right\rVert _{L^{N+2,1}(Q_{t_1})}=\frac{1}{2} ,
	\end{equation}
	then we have
	\begin{equation}\label{c5}
		\left\lVert \left\lvert \nabla u^{\varepsilon}\right\rvert^{\delta(x)} \right\rVert _{L^{\frac{N+2}{N+1},\infty}(Q_{t_1})}\leqslant C_5, 
	\end{equation}
	which imply that
	\begin{equation}\label{tk2tkp}
		\sup_{t\in(0,t_1)}\int_{\Omega}{\left\lvert T_k(u^{\varepsilon}(t))\right\rvert^2 }\mathrm{d}x+\iint_{Q_{t_1}}{\left\lvert \nabla T_k(u^{\varepsilon})\right\rvert^{p(x)} }\mathrm{d}x\mathrm{d}t\leqslant C_6.
	\end{equation}
	Let us divide the time interval $[0,T]$ into a finite number subintervals 
	$[t_0, t_1]$, $[t_1, t_2]$, ..., $[t_{s-1}, t_s]$ ($s=T$), such that for each subinterval, an analogous
	condition \eqref{c4gln} is satisfied,then for each subinterval, we obtain a prior estimate similar to \eqref{tk2tkp}. Putting these estimates together, we conclude that
	\begin{equation}\label{c7}
		\left\| | \nabla u^{\varepsilon}|^{\delta(x)} \right\| _{L^{\frac{N+2}{N+1},\infty}(Q_T)}\leqslant C_7, 
	\end{equation}
	and
	\begin{equation}\label{c8}
		\sup_{t\in (0,T)}\int_{\Omega}{\left\lvert T_k(u^{\varepsilon}(t))\right\rvert^2 }\mathrm{d}x+\iint_{Q_{T}}{\left\lvert \nabla T_k(u^{\varepsilon})\right\rvert^{p(x)} }\mathrm{d}x\mathrm{d}t\leqslant C_8,
	\end{equation}
	where $C_7$ and $C_8$ are constants independent of $\varepsilon$.
	Then inequality \eqref{c8} allows us to prove that
	\begin{equation}\label{tkbnd}
		T_k(u^{\varepsilon})\text{ is bounded in }V
	\end{equation}
	and
	\begin{equation}\label{tkwky}
		T_k(u^{\varepsilon})\rightarrow \sigma _k\text{ weakly in }V.
	\end{equation}
	Moreover, by \eqref{c8}, we have
	\begin{equation}\label{qtuk0}
		\lim_{k\rightarrow +\infty}meas\left\{(x,t)\in Q_T:\left\lvert u^{\varepsilon}\right\rvert >k\right\} =0
	\end{equation}
	\textbf{Step3.} Almost everywhere convergence of $u^{\epsilon}$. 
	
	Making use of $T_k(b_{\varepsilon}(u^{\varepsilon}))\chi _{(0,t)}$ as a test function in \eqref{approximate}, it results
	\begin{equation}
		\begin{aligned}[b]
			& \int_{\Omega}{\overline{T}_k(b_{\varepsilon}(u^{\varepsilon})(t))}dx+\iint_{Q_t}{A_{\varepsilon}(x,t,u^{\varepsilon},\nabla u^{\varepsilon})\nabla T_k(b_{\varepsilon}(u^{\varepsilon}))}dxdt
			\\&+\iint_{Q_t}{H_{\varepsilon}(x,t,\nabla u^{\varepsilon})T_k(b_{\varepsilon}(u^{\varepsilon}))}dxdt
			\\=&\iint_{Q_t}{f^{\varepsilon}T_k(b_{\varepsilon}(u^{\varepsilon}))}dxdt+\int_{\Omega}{\overline{T}_kb_{\varepsilon}(u_0^{\varepsilon}) }dx
		\end{aligned}
	\end{equation}
	for almost every $t\in(0,T)$.
	Using \eqref{as5.1}, \eqref{c7} and the fact that
	\begin{equation}
		\begin{aligned}[b]
			&\iint_{Q_t}{A_{\varepsilon}(x,t,u^{\varepsilon},\nabla u^{\varepsilon})\nabla T_k(b_{\varepsilon}(u^{\varepsilon}))}\mathrm{d}x\mathrm{d}t
			\\=&\iint_{\left\{\left\lvert b_{\varepsilon}(u^{\varepsilon})\right\rvert \leqslant k\right\} } {A_{\varepsilon}(x,t,u^{\varepsilon},\nabla u^{\varepsilon})b'_{\varepsilon}\nabla u^{\varepsilon}}\mathrm{d}x\mathrm{d}t \geqslant 0,
		\end{aligned}
	\end{equation} 
	we obtain
	\begin{equation}
		\begin{aligned}[b]
			&\int_{\Omega}{\overline{T}_k(b_{\varepsilon}(u^{\varepsilon})(t))}\mathrm{d}x 
			\\ \leqslant& k\left(\left\lVert \left\lvert \nabla u^{\varepsilon}\right\rvert^{\delta (x)} \right\rVert _{L^{\frac{N+2}{N+1},\infty}(Q_T)}\left\lVert g(x,t)\right\rVert _{L^{N+1,1}(Q_T)}+\left\lVert f\right\rVert _{L^1(Q_T)}+\left\lVert b_{\varepsilon}(u_0^{\varepsilon})\right\rVert _{L^1(\Omega)} \right)  
			\\ \leqslant& C_{9}k.
		\end{aligned}
	\end{equation}
	Since $\overline{T}_1(s)\geqslant \left\lvert s\right\rvert -1,\forall s\in R$, we deduce from above inequality that
	\begin{equation}\label{c10}
		\int_{\Omega}{\left\lvert b_{\varepsilon}(u^{\varepsilon})\right\rvert }\mathrm{d}x=\int_{\Omega}{\overline{T}_1(b_{\varepsilon}(u^{\varepsilon}))}dx+\left\lvert \Omega\right\rvert \leqslant C_{10},
	\end{equation}
	which imply that
	\begin{equation}\label{qtbk0}
		\lim_{k\rightarrow +\infty} meas\left\{(x,t)\in Q_T:\left\lvert b_{\varepsilon}(u^{\varepsilon})\right\rvert>k \right\}=0.
	\end{equation}
	Consider now  nondecreasing function $ \mathcal{J} _k(s) \in C^2(R)$ such that $\mathcal{J} _k(s)=s$ for $\left\lvert s\right\rvert\leqslant \frac{k}{2}$ and $\mathcal{J} _k(s)=k\,\text{sign}(s)$ for $\left\lvert s\right\rvert\geqslant k$. Taking $\mathcal{J} '_k(u^{\epsilon}) $ as a test function in \eqref{approximate}, then
	\begin{equation}
		\begin{aligned}[b]
			&\frac{\partial \mathcal{J} _k\left(b_{\varepsilon} (u^{\varepsilon}) \right)}{\partial t}-\nabla\cdot\left( \mathcal{J}'_k\left( u ^{\varepsilon}\right) A\left( x,t,T_{\frac{1}{\epsilon}}(u^{\varepsilon}),\nabla u^{\varepsilon} \right) \right) 
			\\ &+\mathcal{J}''_k\left( u^{\varepsilon} \right) A\left( x,t,T_{\frac{1}{\varepsilon}}(u^{\varepsilon}),\nabla u^{\varepsilon} \right) \cdot \nabla u^{\varepsilon}+H_{\varepsilon}\left( x,t,\nabla u^{\varepsilon} \right) \mathcal{J}'_k\left( u^{\varepsilon} \right) 
			\\ =&f^{\varepsilon}\mathcal{J}'_k\left( u^{\varepsilon} \right)\text{ in }D'\left( Q_T \right). 
		\end{aligned}
	\end{equation}
	Due to definition of $b$ and $b_{\varepsilon}$, it is clear that
	\begin{equation}
		\begin{aligned}[b]
			\left\{\left\lvert b_{\varepsilon}(u^{\varepsilon})\right\rvert\leqslant k \right\} \subset \left\{\left\lvert b(u^{\varepsilon})\right\rvert \leqslant k \right\} =\left\{b^{-1}(-k)\leqslant u^{\varepsilon}\leqslant b^{-1}(k)\right\} 
		\end{aligned}
	\end{equation}
	as soon as $\varepsilon <\frac{1}{k}$. Then we otain 
	\begin{equation}
		\left\{\left\lvert b_{\varepsilon}(u^{\varepsilon})\right\rvert \leqslant k\right\} \subset  \left\{\left\lvert u^{\varepsilon}\right\rvert \leqslant k^{\ast}\right\}, 
	\end{equation}
	where $k^*$ is a constant independent of $\varepsilon$. Since $$\nabla \mathcal{J}_k (b_{\varepsilon}(u^{\varepsilon}))=\mathcal{J}_k'(b_{\varepsilon}(u^{\varepsilon}))b'_{\varepsilon}\,\nabla T_{k^{\ast}}(u^{\varepsilon}) a.e\text{ in }Q_T$$
	as soon as $k^{\ast}<\frac{1}{\varepsilon}$,
	by \eqref{as2}, \eqref{tkbnd} and definition of $\mathcal{J} _k$, we have
	\begin{equation}\label{jkbnd}
		\mathcal{J}_k(b_{\varepsilon}(u^{\varepsilon}))\text{ is bounded in }V.
	\end{equation}
	Since $\text{ supp}\mathcal{J}'$ is included in $[-k,k]$, it follows that 
	\begin{equation}
		\begin{aligned}[b]
			&\iint_{Q_T}{\left\lvert \mathcal{J}'_k(b_{\varepsilon}(u^{\varepsilon})) A_{\varepsilon}(x,t,u^{\varepsilon},\nabla u^{\varepsilon})\right\rvert ^{p'(x)}}\mathrm{d}x\mathrm{d}t
			\\ \leqslant& (\left\lVert \mathcal{J}'_k\right\rVert _{L^{\infty}(R)}+1)^{\frac{p^-}{p^--1}}\iint _{Q_T}{(cL(x,t)+c\left\lvert \nabla T_{k^{\ast}}(u^{\varepsilon})\right\rvert^{p(x)-1})^{p'(x)}}\mathrm{d}x\mathrm{d}t
			\\ \leqslant& C_{k_1^*}.
		\end{aligned}
	\end{equation}
	Furthermore, by \eqref{as2}, \eqref{as3.3} and \eqref{tkbnd}, it results
	\begin{equation}
		\begin{aligned}[b]
			&\iint_{Q_T}{\left\lvert \mathcal{J}''_k(b_{\varepsilon}(u^{\varepsilon}))A_{\varepsilon}(x,t,u^{\varepsilon},\nabla u^{\varepsilon})b'_{\varepsilon}(u^{\varepsilon})\nabla u^{\varepsilon}\right\rvert }\mathrm{d}x\mathrm{d}dt
			\\ \leqslant& c b_1 \left\lVert \mathcal{J}''_k \right\rVert _{L^{\infty}(R)}\left(\iint_{Q_T}{\left(\left\lvert L(x,t)\right\rvert +\left\lvert \nabla T_{k^{\ast}}(u^{\varepsilon})\right\rvert^{p(x)-1} \right)^{p'(x)}\mathrm{d}x\mathrm{d}t+1 }\right)^{\frac{1}{p'^-}}
			\\ & \cdot \left\lVert \nabla T_{k^{\ast}}(u^{\varepsilon})\right\rVert _{L^p(x)(Q_T)}
			\\ \leqslant& C_{k_2^*},   
		\end{aligned}
	\end{equation}
	where $C_{k_1^{\ast}} $ and $C_{k_2^{\ast}} $ are comstants independently of $\varepsilon$. By assumption \eqref{as5.1} the term $f^{\varepsilon}\mathcal{J}'_k(b_{\varepsilon}(u^{\varepsilon}))$ is bounded in $L^1(Q_T)$. Finally, from \eqref{ht} and \eqref{c5} it follows that
	\begin{equation}\label{c11}
		\begin{aligned}[b]
			&\iint_{Q_T}{\left\lvert H_{\varepsilon}(x,t,\nabla u^{\varepsilon}) \mathcal{J}'_k(b_ {\varepsilon}(u^{\varepsilon}))\right\rvert }\mathrm{d}x\mathrm{d}t
			\\ \leqslant&\left\lVert \mathcal{J}'_k\right\rVert _{L^{\infty}(R)} \iint_{Q_T}{\left\lvert \nabla u^{\varepsilon}\right\rvert ^{\delta (x)}\left\lvert g(x,t)\right\rvert }\mathrm{d}x\mathrm{d}t
			\\ \leqslant& C_{11}.
		\end{aligned}
	\end{equation}
	As a consequence, we have
	\begin{equation}\label{jtbnd}
		\dfrac{\partial \mathcal{J}_k(b_{\varepsilon}(u^{\varepsilon}))}{\partial t}\text{ is bounded in }V^{\ast}+L^1(Q_T).
	\end{equation}
	The main idea to prove almost everywhere convergence of $u^{\varepsilon}$ is essentially from \cite{Blanchard1998}. Let $r$ be fixed such that $r<inf \left\{(p^+)',\frac{N}{N-1}\right\}$. By \eqref{jtbnd} we deduce that
	\begin{equation}\label{c12}
		\left\lVert \dfrac{\partial \mathcal{J}_k(b_{\varepsilon}(u^{\varepsilon}))}{\partial t}\right\rVert _{L^1(0,T;W^{-1,r}(\Omega))}\leqslant C_{12}. 
	\end{equation}
	From the embedding $V\hookrightarrow  L^{p^-}((0,T);W_0^{1,p(x)}(\Omega))$ and \eqref{jkbnd} it follows that
	\begin{equation}\label{jkubnd}
		\mathcal{J}_k(b_{\varepsilon}(u^{\varepsilon}))\text{ is uniformly bounded in }L^{p^-}((0,T);W_0^{1,p(x)}(\Omega)).
	\end{equation}
	Since $W_0^{1,p(x)}(\Omega)\hookrightarrow \hookrightarrow L^{q(x)}(\Omega)\hookrightarrow W^{-1,r}(\Omega)$ where $q(x)=\frac{N}{N-p(x)}$, \eqref{c12} and \eqref{jkubnd} together with Simon compactness Theorem (see corollary4 of \cite{Simon1986}) imply that  $\left\{\mathcal{J}_k(b_{\varepsilon}(u^{\varepsilon}))\right\} _{\varepsilon}$ is compact in $ L^{p^-}(0,T;L^{p(x)}(\Omega))$, and it converges in measure as $\varepsilon \rightarrow 0$. Recalling \eqref{qtbk0}, for all $\sigma >0$, there holds
	\begin{equation*}
		\begin{aligned}[b]
			meas &\left\{ \left\lvert b_{\varepsilon}(u^{\varepsilon})-b_{\eta}(u^{\eta})\right\rvert >\sigma \right\}\leqslant meas\left\{\left\lvert b_{\varepsilon}(u^{\varepsilon})\right\rvert>\frac{k}{2} \right\}
			\\ &+meas\left\{\left\lvert b_{\eta}(u^{\eta})\right\rvert>\frac{k}{2} \right\}+meas\left\{\left\lvert \mathcal{J}_k(b_{\varepsilon}(u^{\varepsilon}))-\mathcal{J}_k(b_{\eta}(u^{\eta}))\right\rvert>\sigma \right\}.
		\end{aligned}
	\end{equation*}
	Let $k$ sufficiently large, then it follows from \eqref{qtbk0} that
	\begin{equation}\label{bjk0}
		\begin{aligned}[b]
			meas \left\{ \left\lvert b_{\varepsilon}(u^{\varepsilon})-b_{\eta}(u^{\eta})\right\rvert >\sigma \right\} \leqslant \varepsilon +meas\left\{\left\lvert \mathcal{J}_k(b_{\varepsilon}(u^{\varepsilon}))-\mathcal{J}_k(b_{\eta}(u^{\eta}))\right\rvert>\sigma \right\}.
		\end{aligned}
	\end{equation}
	Since  $\left\{\mathcal{J}_k(b_{\varepsilon}(u^{\varepsilon}))\right\}$ converges in measure as $\varepsilon \rightarrow 0$, we deduce that $ \left\{\mathcal{J}_k(b_{\varepsilon}(u^{\varepsilon}))\right\}_{\varepsilon} $ is a cauchy sequence in measure. \eqref{bjk0} indicates that $b_{\varepsilon}(u^{\varepsilon})$ is a cauchy sequence in measure.
	According to the Riesz Representation Theorem, we deduce that for a subsequence, still indexed by $\varepsilon$, $b_{\varepsilon}(u^{\varepsilon})$ converges almost everywhere, as $\varepsilon$ goes to zero, to a measurable function $\varpi $ defined on $Q_T$.
	$b_{\varepsilon}^{-1}$ converges everywhere to $b^{-1}$ due to continuous of $b^{-1}$ on $R$. Then we have
	\begin{equation}\label{ub-1}
		u^{\varepsilon} \rightarrow u=b^{-1}(\varpi )\,\,a.e\text{ in }Q_T,
	\end{equation}
	\begin{equation}\label{bb}
		b_{\varepsilon}(u^{\varepsilon})\rightarrow b(u)\,\,a.e\text{ in }Q_T. 
	\end{equation}
	For every $k>0$, \eqref{tkwky}, \eqref{c10} and \eqref{ub-1} taken together yields that
	\begin{equation}\label{tktkwky}
		T_k(u^{\varepsilon})\rightarrow T_k(u)\text{ weakly in V and }a.e\text{ in }Q_T,
	\end{equation}
	\begin{equation}\label{nablatkly}
		\nabla T_k(u^{\varepsilon})\rightarrow \nabla T_k(u)\text{ weakly in }\left(L^{p(x)}(Q_T)\right)^N, 
	\end{equation}
	Using \eqref{c10} and Fatou's lemma, we have
	\begin{equation}
		\int_{\Omega}{\left\lvert b(u)\right\rvert }\mathrm{d}x \leqslant C_{13}. 
	\end{equation}
	\begin{remark}
		For $0<\varepsilon<\frac{1}{k}$, note that
		$$A(x,t,T_{\frac{1}{\varepsilon}}(u^{\varepsilon}),\nabla T_k(u^{\varepsilon}))=A(x,t,T_k(u^{\varepsilon}),\nabla T_k(u^{\varepsilon}))$$ 
		and
		$$\left\lvert A(x,t,T_k(u^{\varepsilon}),\nabla T_k(u^{\varepsilon}))\right\rvert \leqslant c\left(L(x,t)+\left\lvert \nabla T_k(u^{\varepsilon})\right\rvert^{p(x)-1} \right). $$
		It follows that $A(x,t,T_k(u^{\varepsilon}),\nabla T_k(u^{\varepsilon}))$ is bounded in $\left(L^{p'(x)}(Q_T)\right)^N$
		, as $\varepsilon\rightarrow 0$, for any $k>0$. Then
		\begin{equation}\label{atkwky}
			\begin{aligned}[b]
				A_{\varepsilon}(x,t,T_k(u^{\varepsilon}),\nabla T_k(u^{\varepsilon}))\rightarrow \varphi_k\,\text{ weakly in } (L^{p'(x)}(Q_T))^N. 
			\end{aligned}
		\end{equation}
		Due to growth assumption \eqref{as5.1} on $H$, estimate \eqref{c7} yields that
		\begin{equation}
			\begin{aligned}[b]
				&\iint_{Q_T}  \left\lvert H_{\varepsilon}(x,t,\nabla u^{\varepsilon})\right\rvert \mathrm{d}x\mathrm{d}t\\
				\leqslant& \iint_{Q_T}\left\lvert g(x,t)\right\rvert\left\lvert \nabla u^{\varepsilon}\right\rvert^{\delta(x)}\mathrm{d}x\mathrm{d}t
				\\ \leqslant&\left\lVert g(x,t)\right\rVert _ {L^{N+2,1}(Q_T)}\left\lVert \left\lvert \nabla u^{\varepsilon} \right\rvert^{\delta(x)} \right\rVert _{L^{\frac{N+2}{N+1},\infty}(Q_T)}
				\\ \leqslant& C_{14}.
			\end{aligned}
			\nonumber
		\end{equation}
		It follows that there exists $\varLambda $ belonging to $L^1(Q_T)$ and a subsequence of $H_{\varepsilon}(x,t,\nabla u^{\varepsilon}) $, still indexed by $\varepsilon$, such that
		\begin{equation}\label{hwkyl}
			H_{\varepsilon}(x,t,\nabla u^{\varepsilon})\rightarrow \varLambda \text{ weakly in } L^1(Q_T)
		\end{equation}
		as $\varepsilon \to 0 $ .
	\end{remark}
	
	\begin{lemma}\label{int0}
		The subsequence of $u^{\varepsilon}$ defined in Step 1 satisfies
		\begin{equation}
			\begin{aligned}[b]
				\lim_{n\rightarrow +\infty}\limsup_{\varepsilon \rightarrow 0}\frac{1}{n}\iint_{\left\{ \left( x,t \right) \in Q_T:\left| u^{\varepsilon}\left( x,t \right) \right|< n \right\}}{A\left( x,t,u^{\varepsilon},\nabla u^{\varepsilon} \right) \nabla u^{\varepsilon}\mathrm{d}x\mathrm{d}t}=0.
			\end{aligned}
		\end{equation}
		\begin{proof}
			Using the admissible test function $\frac{T_n(u^{\varepsilon})}{n}$ in \eqref{approximate} and recalling \eqref{thetatk} yield that
			\begin{equation}
				\begin{aligned}[b]
					&\frac{1}{n}  \int_{\Omega}{\Theta _n^{\varepsilon}(u^{\varepsilon}(T))}\mathrm{d}x+\frac{1}{n}\iint_{\left\{\left\lvert u^{\varepsilon}\right\rvert<n\right\}}A_{\varepsilon}(x,t,u^{\varepsilon},\nabla u^{\varepsilon})\nabla u^{\varepsilon}\mathrm{d}x\mathrm{d}t
					\\ \leqslant&   \frac{1}{n}\iint_{Q_T}\left\lvert H_{\varepsilon}(x,t,\nabla u^{\varepsilon})T_n(u^{\varepsilon})\right\rvert+\frac{1}{n}\int_{\Omega}\Theta _n^{\varepsilon}(u^{\varepsilon}_0)dx +\frac{1}{n}\iint_{Q_T}\left\lvert f^{\varepsilon}\right\rvert \cdot \left\lvert T_n(u^{\varepsilon})\right\rvert \mathrm{d}x\mathrm{d}t,
				\end{aligned}
				\nonumber
			\end{equation}
			where $\Theta _n^{\varepsilon}(r)=\int_0^r{ b_{\varepsilon}'(s)T_n(s)}\mathrm{d}s$. By \eqref{aa} and $\Theta _n^{\varepsilon}(u^{\varepsilon})(T)\geqslant 0$, it follows that for $0<\varepsilon<\frac{1}{n}$
			\begin{equation}\label{ahf}
				\begin{aligned}[b]
					&\frac{1}{n}\iint_{\left\{\left\lvert u^{\varepsilon}\right\rvert<n\right\}}A_{\varepsilon}(x,t,u^{\varepsilon},\nabla u^{\varepsilon})\nabla u^{\varepsilon}\mathrm{d}x\mathrm{d}t
					\\ \leqslant& \frac{1}{n}\iint_{Q_T}\left\lvert H_{\varepsilon}(x,t,\nabla u^{\varepsilon})T_n(u^{\varepsilon})\right\rvert \mathrm{d}x\mathrm{d}t
					+\frac{1}{n}\int_{\Omega}\Theta _n^{\varepsilon}(u^{\varepsilon}_0)\mathrm{d}x +\frac{1}{n}\iint_{Q_T}\left\lvert f^{\varepsilon}\right\rvert \left\lvert T_n(u^{\varepsilon})\right\rvert \mathrm{d}x\mathrm{d}t.
				\end{aligned}
			\end{equation}
			Now we handle $H$. Lebesgue's dominated convergence theorem and \eqref{ub-1} implies that 
			$$ T_n(u^{\varepsilon})\rightarrow T_n(u) \text{ and weak-}^*\text{in } L^{\infty}(Q_T) .$$
			Then \eqref{hwkyl} and Egorov's theorem allow us to prove that 
			\begin{equation}
				\lim_{\varepsilon \rightarrow 0}\iint_{Q_T}{H_{\varepsilon}\left( x,t,\nabla u^{\varepsilon} \right) T_n\left( u^{\varepsilon} \right) \mathrm{d}x\mathrm{d}t=}\iint_{Q_T}{\varLambda \,\,T_n\left( u \right)}\mathrm{d}x\mathrm{d}t.
			\end{equation}
			From finite almost everywhere of $u$ in $Q_T$ and Lebesgue's dominated convergence theorem it follows that
			\begin{equation}
				\lim_{n\rightarrow \infty}\lim_{\varepsilon \rightarrow 0}\,\,\frac{1}{n}\iint_{Q_T}{H_{\varepsilon}\left( x,t,\nabla u^{\varepsilon} \right) T_n\left( u^{\varepsilon} \right) \,\,\mathrm{d}x\mathrm{d}t=}\,0.
			\end{equation}
			Similar arguments lead to
			\begin{equation}
				\lim_{n\rightarrow \infty} \,\lim_{\varepsilon \rightarrow 0} \,\,\frac{1}{n}\iint_{Q_T}f^{\varepsilon}T_n(u^{\varepsilon}) \mathrm{d}x\mathrm{d}t=0,
			\end{equation}
			\begin{equation}\label{nll0}
				\lim_{n\rightarrow \infty}\,\lim_{\varepsilon \rightarrow 0} \frac{1}{n}\int_{\varOmega}\Theta^{\varepsilon}_n(u^{\varepsilon}_0) \mathrm{d}x=0.
			\end{equation}
			Gathering \eqref{ahf}-\eqref{nll0} we get
			\begin{equation}
				\lim_{n\rightarrow \infty} \,\limsup_{\varepsilon \rightarrow 0} \frac{1}{n}\iint_{\left\{ \left\lvert u^{\varepsilon}\right\rvert< n  \right\}} A_{\varepsilon}(x,t,u^{\varepsilon},\nabla u^{\varepsilon})\nabla u^{\varepsilon}\mathrm{d}x\mathrm{d}t=0.
			\end{equation}
		\end{proof}
	\end{lemma}
	\textbf{Step4.} Convergence of thegradient.

	This step is devoted to introduce  a time regularization of the function $T_k(u)$ for fixed $k>0$. We denote this regularized function to $T_k(u)$ by $(T_k(u))_{\mu}$ with $\mu >0$. It is defined as the unique solution 
	$(T_k(u))_{\mu}\in V\cap L^{\infty}(Q_T)$ of the equation
	$$\partial_t (T_k(u))_{\mu}+\mu((T_k(u))_{\mu}-T_k(u))=0$$
	with the initial condition $(T_k(u))_{\mu}\mid _{t=0}=u^{\mu}_0 \,in \,\,\Omega $ and $u_0^{\mu} \in W_0^{1,p(x)}(\varOmega)\cap L^{\infty}(\varOmega)$.
	Note  that this function has the following properties\cite{Blanchard2001}\cite{Zhang2010}
	\begin{equation}\label{4eqs}
		\begin{cases}
			\partial_t(T_k(u))_{\mu}\in V\cap L^{\infty}(Q_T)\\
			\left\lVert (T_k(u))_{\mu}\right\rVert _{L^{\infty}(Q_T)}\leqslant k \\
			(T_k(u))_{\mu}\rightarrow T_k(u)\,a.e\text{ in }Q_T \text{ and weak-}^*\text{in } L^{\infty}(Q_T)\text{ as }\mu \rightarrow \infty\\
			\nabla (T_k(u))_{\mu}\rightarrow \nabla T_k(u)\text{ strongly in } (L^{p(x)}(Q_T))^N\text{ as }\mu\rightarrow \infty.\\
		\end{cases}
	\end{equation}
	The definition of $(T_k(u))_{\mu}$ fo $\mu>0$ and fixed k makes it possible establish the following lemma.
	\begin{lemma}\label{inner0}
		\cite{Blanchard2001}\cite{Li2022}Let $k\geqslant 0$ be fixed and $S$ be a increasing $C^{\infty}(R)$-function such that $S(r)=r\,for\,\left\lvert r \right\rvert \leqslant k$ and $supp S'$ is compact. Then
		\begin{equation}
			\liminf_{\mu \rightarrow \infty} \underset{\varepsilon \rightarrow 0}{\lim} \int_0^T{\left\langle \dfrac{\partial S(u^{\varepsilon})}{\partial t}, (T_k(u^{\varepsilon})-(T_k(u))_{\mu})\right\rangle} dt \geqslant 0,
		\end{equation}
		where $\left\langle \cdot , \cdot\right\rangle$ denote the duality pairing between $L^1(\Omega)+W^{-1,p'(x)}(\Omega)$ and $L^{\infty}(\Omega)\cap V$. 
	\end{lemma}
	Let $S_n$ be a sequence of increasing $C^{\infty}$-functions such that
	$$S_n(r)=r,\text{ for }\left\lvert r\right\rvert\leqslant n,\text{ supp }S'_n\subset [-2n,2n]\text{ and }\left\lVert S''_n\right\rVert _{L^{\infty}(R)}\leqslant \frac{1}{n}$$  
	for any $n\geqslant 1$.
	We consider the test function $S'_n(u^{\varepsilon})(T_k(u^{\varepsilon})-(T_k(u))_{\mu})\,for\,n\geqslant1$ and $\mu>0$ in \eqref{approximate}. For fixed $k\geqslant0$, we define $W^{\varepsilon}_{\mu}=T_k(u^{\varepsilon})-(T_k(u))_{\mu}$
	and by integrating over $(0,t)$ and then over $(0,T)$, we get
	\begin{equation}
		\begin{aligned}[b]
			&\iint_{Q_T}  {A_{\varepsilon}\left( x,t,u^{\varepsilon},\nabla u^{\varepsilon} \right) S'_n\left( u^{\varepsilon} \right) \nabla W_{\mu}^{\varepsilon}}\,\,dxdt
			\\  =& - \int_0^T\left\langle \frac{\partial u^{\varepsilon}}{\partial t},S'_n(u^{\varepsilon})W^{\varepsilon}_{\mu}\right\rangle dt
			\\ & -\iint_{Q_T}{A_{\varepsilon}\left( x,t,u^{\varepsilon},\nabla u^{\varepsilon} \right) S''\left( u^{\varepsilon} \right) \nabla u^{\varepsilon} W_{\mu}^{\varepsilon}}\,\,dxdt
			\\ &-\iint_{Q_T}{H_{\varepsilon}\left( x,t,\nabla u^{\varepsilon},\nabla u^{\varepsilon} \right) S'\left( u^{\varepsilon} \right) W_{\mu}^{\varepsilon}}\,\,dxdt
			\\ &+\iint_{Q_T}f^{\varepsilon}S'_n(u^{\varepsilon})W^{\varepsilon}_{\mu} dxdt
			\\ =&-I_1-I_2-I_3+I_4.
		\end{aligned}
	\end{equation}
	Now we show the limits of $I_1\backsim I_4$ when $\varepsilon$ goes to zero and $n$, $\mu $ tend to infinity respectly \cite{Blanchard2001}.
	\textbf{Limit of} $I_1$.
	Since the function $ S_n$ belongs $C^{\infty}(R)$ and is increasing, we have $S_n(r)=r$,for $\left\lvert r \right\rvert\leqslant k \text{ and } k \leqslant n $. By virtue ofthe definition of $W^{\varepsilon}_{\mu} $ lemma \eqref{inner0} applies with $S=S_n$ for fixed $k \leqslant n$. Hence, it follows that
	\begin{equation}\label{liminfinner0}
		\liminf_{\mu \rightarrow \infty}\lim_{\varepsilon \rightarrow 0}\int_0^T{\left\langle \frac{\partial S(u^{\varepsilon})}{\partial t}, W^{\varepsilon}_{\mu}\right\rangle} \mathrm{d}t \geqslant 0.
	\end{equation}
	\textbf{Limit of} $I_2$.
	In view of the definition of $S_n$, we have
	\begin{equation}
		\begin{aligned}[b]
			\left\lvert I_2 \right\rvert 
			&=\left\lvert \iint_{Q_T}{A_{\varepsilon}(x,t,u^{\varepsilon}\nabla u^{\varepsilon})}S''_n(u^{\varepsilon})W^{\varepsilon}_{\mu}\nabla u^{\varepsilon} \mathrm{d}x\mathrm{d}t\right\rvert
			\\&\leqslant T \left\lVert S''_n \right\rVert _{L^{\infty}(R)} \left\lVert W^{\varepsilon}_{\mu}\right\rVert _{L^{\infty}(Q_T)}\iint_{\left\{n\leqslant \left\lvert u^{\varepsilon}\right\rvert \leqslant 2n\right\} }A(x,t,u^{\varepsilon},\nabla u^{\varepsilon})\cdot \nabla u^{\varepsilon}\mathrm{d}x\mathrm{d}t
		\end{aligned}
		\nonumber
	\end{equation}
	for any $n\geqslant 1$ and $0<\varepsilon<\frac{1}{2n}$ and for any $\mu >0$. Using $\left\lVert S''_n\right\rVert _{L^{\infty}(R)}\leqslant\frac{1}{n}$ and lemma \eqref{int0}, it is possible to establish
	\begin{equation}\label{llll2}
		\lim_{n\rightarrow \infty}\limsup_{\mu \rightarrow \infty}\,\limsup_{\varepsilon \rightarrow 0}\,I_2\,=0.
	\end{equation}
	\textbf{Limit of} $I_3$.
	Firstly, \eqref{tktkwky} and definition of $W^{\varepsilon}_{\mu}$ imply that for fixed $\mu>0$
	\begin{equation}
		W^{\varepsilon}_{\mu} \rightarrow T_k(u)-(T_k(u))_{\mu}\,\text{ weakly in }L^{p^-}((0,T);W_0^{1,p(x)}(\Omega))
	\end{equation}
	as $\varepsilon \rightarrow 0$.
	Furthmore, \eqref{4eqs} leads to
	\begin{equation}\label{w2k}
		\left\lVert W^{\varepsilon}_{\mu}\right\rVert _{L^{\infty}(Q_T)}\leqslant 2k. 
	\end{equation} 
	Therefore, we deduce that for fixed $\mu>0$
	\begin{equation}\label{wtktk}
		W^{\varepsilon}_{\mu}\rightarrow T_k(u)-(T_k(u))_{\mu}\,\,\, a.e\,\text{ in }\,Q_T\,\text{ and weak-}^*\text{in }L^{\infty}(Q_T).
	\end{equation}
	\\Since $S'_n$ is smooth and bounded, \eqref{ub-1}, \eqref{w2k} and \eqref{wtktk} lead to
	\begin{equation}\label{stk}
		S'_n(u^{\varepsilon})W^{\varepsilon}_{\mu}\rightarrow S'(u)(T_k(u)-(T_k(u))_{\mu})\,\,a.e\,\text{ in }\,Q_T\text{ and weak-}^*\text{in}\,L^{\infty}(Q_T) .
	\end{equation}
	In view of the weak convergence of $H_{\varepsilon}(x,t,\nabla T_{2n}(u^{\varepsilon}))\,in\,L^1(Q_T)$, \eqref{stk} and Egorov's theorem allows us to prove that
	\begin{equation}\label{hsstk}
		\begin{aligned}[b]
			\lim_{\varepsilon\rightarrow0}\iint_{Q_T}{H_{\varepsilon}(x,t,\nabla T_{2n}(u^{\varepsilon}))S'_n(u^{\varepsilon})W^{\varepsilon}_{\mu}}\mathrm{d}x\mathrm{d}t
			=\iint_{Q_T}{\varLambda\, S'_n(u)(T_k(u)-(T_k(u))_{\mu})}\mathrm{d}x\mathrm{d}t.
		\end{aligned}
	\end{equation}
	Appealing to \eqref{4eqs} and passing to the limit ad $\mu\rightarrow\infty$ in \eqref{hsstk} makes it possible to conclude that
	\begin{equation}\label{lll3}
		\lim_{\mu\rightarrow\infty}\lim_{\varepsilon\rightarrow0}\,I_3=0.
	\end{equation}
	\textbf{Limit of} $I_4$. In view of \eqref{fl}, \eqref{ub-1}, \eqref{tktkwky} and \eqref{wtktk}, Lebesgue's convergence theorem imply that for any $\mu>0$ and any $n\geqslant1$
	\begin{equation}\label{fsfstk}
		\begin{aligned}[b]
			\lim_{\varepsilon\rightarrow0}\iint_{Q_T}{{f^{\varepsilon}S'_n(u^{\varepsilon})W^{\varepsilon}_{\mu}}}\mathrm{d}x\mathrm{d}t
			=\iint_{Q_T}{fS'_n(u)(T_k(u)-(T_k(u))_{\mu})}\mathrm{d}x\mathrm{d}t.
		\end{aligned}
	\end{equation}
	Now for fixed $n\geqslant1$, using \eqref{4eqs} and \eqref{fsfstk} makes it possible to pass to the limit as $\mu$ tend to $\mu\rightarrow\infty$ in the above equaction to obtain
	\begin{equation}\label{lll4}
		\lim_{\mu\rightarrow\infty}\lim_{\varepsilon\rightarrow0}\, I_4\,=\,0.
	\end{equation}
	Due to \eqref{liminfinner0}, \eqref{llll2}, \eqref{lll3} and \eqref{lll4}, we have
	\begin{equation}
		\begin{aligned}[b]
			\lim_{n\rightarrow \infty}\limsup_{\mu \rightarrow \infty}\limsup_{\varepsilon \rightarrow 0}
			\iint_{Q_T}{S'_n(u^{\varepsilon})A_{\varepsilon}(x,t,u^{\varepsilon},\nabla u^{\varepsilon})(\nabla T_k(u^{\varepsilon})-\nabla (T_k(u))_{\mu})}\mathrm{d}x\mathrm{d}t \leqslant 0,
		\end{aligned}
	\end{equation}
	By definition of $S'_n$, we have
	$$S'_n(u^{\varepsilon})A_{\varepsilon}(x,t,u^{\varepsilon},\nabla u^{\varepsilon})\nabla T_k(u^{\varepsilon})=A(x,t,u^{\varepsilon},\nabla u^{\varepsilon})\nabla T_k(u^{\varepsilon})\,\text{ for } k\leqslant\frac{1}{\varepsilon}\,and\,k\leqslant n.$$
	Then the above inequality implies that for $k\leqslant n $
	\begin{equation}\label{lalllqt}
		\begin{aligned}[b]
			&\limsup_{\varepsilon\rightarrow0}\iint_{Q_T}A_{\varepsilon}(x,t,u^{\varepsilon},\nabla u^{\varepsilon})\nabla T_k(u^{\varepsilon})\,\mathrm{d}x\mathrm{d}t
			\\ \leqslant& \lim_{n\rightarrow \infty}\limsup_{\mu \rightarrow \infty}\limsup_{\varepsilon \rightarrow 0}
			\iint_{Q_T}{S'_n(u^{\varepsilon})A_{\varepsilon}(x,t,u^{\varepsilon},\nabla u^{\varepsilon}){\nabla (T_k(u))_{\mu}}}\,\mathrm{d}x\mathrm{d}t. 
		\end{aligned}
	\end{equation}
	On the other hand, for $\varepsilon\leqslant \frac{1}{2n}$, we have
	$$ S'_n(u^{\varepsilon})A_{\varepsilon}(x,t,u^{\varepsilon},\nabla u^{\varepsilon})=S'_n(u^{\varepsilon})A(x,t,T_{2n}(u^{\varepsilon}),\nabla T_{2n}(u^{\varepsilon}))\,a.e\text{ in }Q_T.$$
	Using Remark 3.1 it follows that for fixed $n\geqslant1,\,as\, \varepsilon\rightarrow0$,
	\begin{equation}
		\begin{aligned}[b]
			S'_n(u^{\varepsilon})A_{\varepsilon}(x,t,u^{\varepsilon},\nabla u^{\varepsilon})\rightarrow S'_n(u)\varphi _{2n}\,\, \text{ weakly in }L^{p'(x)}(Q_T).
		\end{aligned}
	\end{equation}
	The strongly convergence of $(T_k(u))_{\mu}$ to $T_k(u)$ in V, as $\mu\rightarrow \infty$, makes it possible to conclude that
	\begin{equation}\label{llsask}
		\begin{aligned}[b]
			&\lim_{\mu\rightarrow\infty}\lim_{\varepsilon\rightarrow0}\iint_{Q_T}S'_n(u^{\varepsilon})A_{\varepsilon} (x,t,u^{\varepsilon},\nabla u^{\varepsilon})\nabla (T_k(u))_{\mu}\mathrm{d}x\mathrm{d}t
			\\ =&\iint_{Q_T}{S'_n(u)\varphi_{2n}\nabla T_k(u)}\mathrm{d}x\mathrm{d}t
			\\ =&\iint_{Q_T}{\varphi_{2n}\nabla T_k(u)}\mathrm{d}x\mathrm{d}t
		\end{aligned}
	\end{equation}
	as soon as $k\leqslant n$. Now we claim that
	\begin{equation}\label{nwct}
		\begin{aligned}[b]
			\lim_{\varepsilon\rightarrow0}\iint_{Q_T}{[A(x,t,T_k(u^{\varepsilon}),\nabla T_k(u^{\varepsilon}))-A(x,t,T_k(u^{\varepsilon}),\nabla T_k(u))]} [\nabla T_k(u^{\varepsilon})-\nabla T_k(u)] \mathrm{d}x\mathrm{d}t=0
		\end{aligned}
	\end{equation}
	and
	\begin{equation}\label{and}
		\varphi_k=A(x,t,T_k(u),\nabla T_k(u))\,\,\,\,a.e\text{ in }Q_T.
	\end{equation}
	In fact, for $k\leqslant n$, we have
	$$ A(x,t,T_{2n}(u^{\varepsilon}),\nabla T_{2n}(u^{\varepsilon}))\chi _{\left\{\left\lvert u^{\varepsilon}\right\rvert<k \right\} }=A(x,t,T_k(u^{\varepsilon}),\nabla T_k(u^{\varepsilon}))\chi _{\left\{\left\lvert u^{\varepsilon}\right\rvert<k \right\} }\,a.e\text{ in }Q_T.$$
	Using \eqref{ub-1} and \eqref{atkwky}, we have
	\begin{equation}\label{2nkqt}
		\varphi_{2n}\chi_{\left\{\left\lvert u\right\rvert<k \right\} }=\varphi_k\chi_{\left\{\left\lvert u\right\rvert<k \right\} }\,\,a.e\text{ in }Q_T\backslash \left\{\left\lvert u\right\rvert=k \right\} \text{ for }k\leqslant n
	\end{equation}
	as $\varepsilon\rightarrow 0$.
	As a consequence of \eqref{2nkqt}, we have for $k\leqslant n$
	\begin{equation}\label{tktkqt}
		\varphi_{2n}\nabla T_k(u)=\varphi_k \nabla T_k(u)\,\,a.e\text{ in }Q_T.
	\end{equation}
	Recalling \eqref{lalllqt}, \eqref{llsask} and \eqref{tktkqt} makes it possible to conclede that
	\begin{equation}\label{att}
		\begin{split}
			\limsup_{\varepsilon\rightarrow0}\iint_{Q_T}A(x,t,u^{\varepsilon},\nabla T_k(u^{\varepsilon}))\nabla T_k(u^{\varepsilon})\mathrm{d}x\mathrm{d}t
			\leqslant \iint_{Q_T}{\varphi_k \nabla T_k(u)}\mathrm{d}x\mathrm{d}t.
		\end{split}
	\end{equation}
	Let $k\geqslant 0$ be fixed, by \eqref{as3.2}, we have 
	\begin{equation}\label{atatt}
		\begin{aligned}[b]
			\lim_{\varepsilon\rightarrow0}\iint_{Q_T}{[A(x,t,T_k(u^{\varepsilon}),\nabla T_k(u^{\varepsilon}))-A(x,t,T_k(u^{\varepsilon}),\nabla T_k(u))]}[\nabla T_k(u^{\varepsilon})-\nabla T_k(u)]\mathrm{d}x\mathrm{d}t\geqslant0.
		\end{aligned}
	\end{equation}
	Furthermore, for $\forall \psi \in (L^{p(x)}(Q_T))^N$ by \eqref{as3.3} and \eqref{ub-1} we have
	$$ A_{\varepsilon}(x,t,T_k(u^{\varepsilon}),\psi )\rightarrow A(x,t,T_k(u),\psi )\, a.e\text{ in }Q_T$$
	as $\varepsilon \rightarrow 0$, and 
	$$\left\lvert A(x,t,T_k(u^{\varepsilon}),\psi )\right\rvert\leqslant \left[L(x,t)+\left\lvert\psi  \right\rvert^{p(x)-1} \right] $$
	uniformly for all $\varepsilon<\frac{1}{k}$. It follows that
	\begin{equation}\label{aasgy}
		A_{\varepsilon}(x,t,T_k(u^{\varepsilon}),\psi )\rightarrow A(x,t,T_k(u),\psi )\text{ strongly in }\left(L^{p'(x)}(Q_T)\right)^N.  
	\end{equation}
	Since 
	\begin{equation}
		\begin{aligned}[b]
			\lim_{\varepsilon\rightarrow0}\iint_{Q_T}{[A(x,t,T_k(u^{\varepsilon}),\nabla T_k(u^{\varepsilon}))-A(x,t,T_k(u^{\varepsilon}),\nabla T_k(u))]}[\nabla T_k(u^{\varepsilon})-\nabla T_k(u)]\mathrm{d}x\mathrm{d}t\geqslant0.
		\end{aligned}
		\nonumber
	\end{equation}
	In view of \eqref{tktkwky}, \eqref{att}, \eqref{atatt} and \eqref{aasgy}, we conclude that
	\begin{equation}
		\begin{aligned}[b]
			\lim_{\varepsilon\rightarrow0}\iint_{Q_T}{[A(x,t,T_k(u^{\varepsilon}),\nabla T_k(u^{\varepsilon}))-A(x,t,T_k(u^{\varepsilon}),\nabla T_k(u))]}[\nabla T_k(u^{\varepsilon})-\nabla T_k(u)]\mathrm{d}x\mathrm{d}t=0.
		\end{aligned}
		\nonumber
	\end{equation}
	Then, above equaction implies that 
	\begin{equation}\label{latt}
		\lim_{\varepsilon \rightarrow 0} \iint_{Q_T}{A(x,t,T_k(u^{\varepsilon}),\nabla T_k(u^{\varepsilon}))\nabla T_k(u^{\varepsilon})\mathrm{d}x\mathrm{d}t}=\iint_{Q_T}{\varphi _k \nabla T_k(u)\mathrm{d}x\mathrm{d}t}
	\end{equation}
	For $\forall \rho \in (C_0^{\infty}(Q_T))^N$ and $m>0$, in view of \eqref{as3.2}, \eqref{tktkwky}, \eqref{atatt}, \eqref{atkwky}, \eqref{aasgy} and \eqref{latt}, we obtain that
	\begin{equation}
		\begin{aligned}[b]
			0\leqslant&\lim_{\varepsilon \rightarrow 0}\iint_{Q_T}[A(x,t,T_k(u^{\varepsilon}),\nabla T_k(u^{\varepsilon}))-A(x,t,T_k(u^{\varepsilon}),\nabla T_k(u)-m\rho)][\nabla T_k(u^{\varepsilon})-(\nabla T_k(u)-m\rho)]
			\\=&\lim_{\varepsilon \rightarrow 0}\iint_{Q_T}{A(x,t,T_k(u^{\varepsilon}),\nabla T_k(u^{\varepsilon}))\cdot\nabla T_k(u^{\varepsilon})\,\mathrm{d}x\mathrm{d}t }
			\\&\;\;\;\;-\lim_{\varepsilon \rightarrow 0}\iint_{Q_T}{A(x,t,T_k(u^{\varepsilon}),\nabla T_k(u^{\varepsilon}))\cdot(\nabla T_k(u)-m\rho)\,\mathrm{d}x\mathrm{d}t }
			\\&\;\;\;\;\;-\lim_{\varepsilon \rightarrow 0}\iint_{Q_T}{A(x,t,T_k(u^{\varepsilon}),\nabla T_k(u)-m\rho)\cdot[\nabla T_k(u^{\varepsilon})-(\nabla T_k(u)-m\rho)]\,\mathrm{d}x\mathrm{d}t}
			\\=&\iint_{Q_T}{\varphi _k\cdot \nabla T_k(u)\,\mathrm{d}x\mathrm{d}t}-\iint_{Q_T}{\varphi _k\cdot (\nabla T_k(u)-m\rho)\,\mathrm{d}x\mathrm{d}t}
			\\&-m\iint_{Q_T}{A(x,t,T_k(u),\nabla T_k(u)-m\rho)\cdot \rho \, \mathrm{d}x\mathrm{d}t},
		\end{aligned}
		\nonumber
	\end{equation}
	it follows that
	\begin{equation}
		\iint_{Q_T}{[\varphi _k-A(x,t,T_k(u),\nabla T_k(u)-m\rho)]\cdot \rho \,\mathrm{d}x\mathrm{d}t}\geqslant0.
		\nonumber
	\end{equation}
	Let $m\rightarrow 0$ in above equation, we have
	\begin{equation}
		\iint_{Q_T}{[\varphi _k-A(x,t,T_k(u),\nabla T_k(u))]\cdot \rho \,\mathrm{d}x\mathrm{d}t}\geqslant0.
		\nonumber
	\end{equation}
	Replacing by $\rho$ by $-\rho$, we actually have
	\begin{equation}
		\iint_{Q_T}{[\varphi _k-A(x,t,T_k(u),\nabla T_k(u))]\cdot \rho \,\mathrm{d}x\mathrm{d}t}\leqslant0,
		\nonumber
	\end{equation}
	that is 
	\begin{equation}
		\iint_{Q_T}{[\varphi _k-A(x,t,T_k(u),\nabla T_k(u))]\cdot \rho \,\mathrm{d}x\mathrm{d}t}=0.
		\nonumber
	\end{equation}
	Thus we obtain
	$$\varphi _k=A(x,t,T_k(u),\nabla T_k(u)).$$
	In summary,it follows that \eqref{nwct} and \eqref{and} holds true. Therefore,for any $k\geqslant0$, we conclede that
	\begin{equation}\label{atattt0}
		\begin{aligned}[b]
			[A(x,t,T_k(u^{\varepsilon}),\nabla T_k(u^{\varepsilon}))-A(x,t,T_k(u^{\varepsilon}),\nabla T_k(u))]\left(\nabla T_k(u^{\varepsilon})-\nabla T_k(u)\right)\rightarrow 0 
		\end{aligned}
	\end{equation}
	strongly in $L^1(Q_T)$ as $\varepsilon \rightarrow 0$. 
	Note \eqref{tktkwky} and \eqref{atattt0}, if we use lemma 5 in \cite{Boccardo1988} or lemma 4.6 in \cite{Yazough2018}, then
	\begin{equation}\label{tktksgy}
		T_k(u^{\varepsilon})\rightarrow T_k(u)\,\text{ strongly in }V.
	\end{equation}
	For all $\sigma >0$, there holds 
	\begin{equation}
		\begin{aligned}[b]
			meas &\left\{(x,t)\in Q_T:\left\lvert \nabla u^{\varepsilon}-\nabla u\right\rvert >\sigma\right\} \leqslant meas\left\{(x,t)\in Q_T:\left\lvert u^{\varepsilon}\right\rvert>k \right\} 
			\\ & +meas \left\{(x,t)\in Q_T:\left\lvert u\right\rvert>k\right\} +meas\left\{(x,t)\in Q_T:\left\lvert \nabla T_k(u^{\varepsilon})-\nabla T_k(u)\right\rvert >\sigma\right\}.
			\nonumber
		\end{aligned}
	\end{equation}
	Let $k$ sufficiently large, by \eqref{qtuk0} and above limit, we get
	\begin{equation}\label{uuqt}
		\nabla u^{\varepsilon}\rightarrow \nabla u\,a.e\text{ in }Q_T.
	\end{equation}
	Furthermore, by \eqref{nablatkly}, \eqref{atkwky}, \eqref{nwct}, \eqref{and} and \eqref{aasgy}, we have
	\begin{equation}
		\begin{aligned}[b]
			A(x,t,T_k(u^{\varepsilon}),\nabla T_k(u^{\varepsilon}))\nabla T_k(u)\rightarrow A(x,t,T_k(u),\nabla T_k(u))\nabla T_k(u)\text{ weakly in }L^1(Q_T),
			\\A(x,t,T_k(u^{\varepsilon}),\nabla T_k(u))\nabla T_k(u^{\varepsilon})\rightarrow A(x,t,T_k(u),\nabla T_k(u))\nabla T_k(u)\text{ weakly in }L^1(Q_T),
			\\A(x,t,T_k(u^{\varepsilon}),\nabla T_k(u))\nabla T_k(u)\rightarrow A(x,t,T_k(u),\nabla T_k(u))\nabla T_k(u)\text{ weakly in }L^1(Q_T).
		\end{aligned}
		\nonumber
	\end{equation}
	Using the above convergence results in \eqref{atattt0} shows that for any $k\geqslant 0$
	\begin{equation}\label{atatqt}
		\begin{aligned}[b]
			&A(x,t,T_k(u^{\varepsilon}),\nabla T_k(u^{\varepsilon}))\nabla T_k(u^{\varepsilon})\rightarrow A(x,t,T_k(u),\nabla T_k(u))\nabla T_k(u)\text{ weakly in }L^1(Q_T)\text{ as }\varepsilon \rightarrow 0.
		\end{aligned}
	\end{equation}
	\textbf{Step5.} Passing to the limit.\\
	Let $\phi \in C_0^{\infty}(Q_T), S\in W^{2,\infty}(R)$ such that $\text{supp} S'\subset [-k,k]$ for some $k>0$ and let S be a function of piecewise $C^1$. Taking advantage of $S'(u^{\varepsilon})\phi $ as a test function in \eqref{approximate}, we have
	\begin{equation}\label{long}
		\begin{aligned}[b]
			\frac{\partial \Theta _S^{\varepsilon}(u^{\varepsilon})}{\partial t}-\nabla\cdot (S'(u^{\varepsilon})A_{\varepsilon}(x,t,u^{\varepsilon},\nabla u^{\varepsilon}))&+S''(u^{\varepsilon})A_{\varepsilon}(x,t,u^{\varepsilon},\nabla u^{\varepsilon})\nabla u^{\varepsilon}
			\\ &+H_{\varepsilon}(x,t,\nabla u^{\varepsilon})S'(u^{\varepsilon})=fS'(u^{\varepsilon})\,\text{ in }D'(Q_T),
		\end{aligned}
	\end{equation}
	where $\Theta _S^{\varepsilon}(z)=\int_0^z{b_{\varepsilon}'(s)S'(s)}ds.$ It follows from \eqref{ub-1} and characteristic of $S$ that
	\begin{equation}\label{first}
		\Theta_{S}^{\varepsilon}(u^{\varepsilon})\rightarrow \Theta_{S}(u)\,a.e\,\text{ in }\,Q_T\text{ and weak}^*\text{ in } L^{\infty}(Q_T),
	\end{equation}
	\begin{equation}\label{second}
		S'(u^{\varepsilon})\rightarrow S'(u)\,a.e\,\text{ in }\,Q_T\text{ and weak}^*\text{ in }L^{\infty}(Q_T),
	\end{equation}
	\begin{equation}\label{third}
		S''(u^{\varepsilon})\rightarrow S''(u)\,a.e\,\text{ in }\,Q_T\text{ and weak}^*\text{ in }L^{\infty}(Q_T).
	\end{equation}
	Then by \eqref{atatqt}, we have
	\begin{equation}\label{stst}
		\dfrac{\partial \Theta _S^{\varepsilon}(u^{\varepsilon})}{\partial t}\rightarrow \dfrac{\partial  \Theta _S(u)}{\partial t} \,\text{ in }\,D'(Q_T).
	\end{equation}
	By $\text{supp}S',S''\subset [-k,k]$, we obtain
	\begin{equation}
		S'(u^{\varepsilon})A(x,t,T_{\frac{1}{\varepsilon}}(u^{\varepsilon}),\nabla u^{\varepsilon})=S'(u^{\varepsilon})A(x,t,T_k(u^{\varepsilon}),\nabla T_k(u^{\varepsilon})),
	\end{equation}
	\begin{equation}
		\begin{aligned}[b]
			S''(u^{\varepsilon})A(x,t,T_{\frac{1}{\varepsilon}}(u^{\varepsilon}),\nabla u^{\varepsilon})\cdot\nabla u^{\varepsilon}= S''(u^{\varepsilon})A(x,t,T_k(u^{\varepsilon}),\nabla T_k(u^{\varepsilon}))\cdot \nabla T_k(u^{\varepsilon})
		\end{aligned}
	\end{equation}
	for $0<\varepsilon<\frac{1}{k}$, and almost everywhere $(x,t)\in Q_T$. Recalling \eqref{atkwky}, \eqref{and}, \eqref{long} and \eqref{third} makes it possible to conclude that
	\begin{equation}
		\begin{aligned}[b]
			S'(u^{\varepsilon})&A(x,t,T_K(u^{\varepsilon}),\nabla T_k(u^{\varepsilon}))\rightarrow S'(u)A(x,t,T_k(u),\nabla T_k(u))
		\end{aligned}
	\end{equation}
	weakly in $(L^{p'(x)}(Q_T))^N$.
	Moreover, \eqref{tktksgy}, \eqref{first} and \eqref{stst} imply that
	\begin{equation}
		\begin{aligned}[b]
			S''(u^{\varepsilon})&A(x,t,T_k(u^{\varepsilon}),\nabla T_k(u^{\varepsilon}))\cdot \nabla T_k(u^{\varepsilon})\rightarrow S''(u)A(x,t,T_k(u),\nabla T_k(u))\cdot\nabla T_k(u)
		\end{aligned}
	\end{equation}
	weakly in $L^1(Q_T)$.
	Recalling that $\varLambda $ is the weak limit of $ H_{\varepsilon}(x,t,\nabla u^{\varepsilon})$, since \eqref{uuqt} implies thatwe have
	$$\varLambda=H(x,t,\nabla u), \,a.e\text{ in }Q_T.$$
	It is suffice to note that
	\begin{equation}
		H_{\varepsilon }(x,t,\nabla u^{\varepsilon})S'(u^{\varepsilon})\rightarrow H(x,t,\nabla u)S'(u)\,\text{ weakly in }L^1(Q_T).
		\nonumber
	\end{equation}
	Using the strong convergence of $f^{\varepsilon}$ to $f\,in\,L^1(Q_T)$, we have
	\begin{equation}
		f^{\varepsilon}S'(u^{\varepsilon})\rightarrow fS'(u)\text{ strongly in }L^1(Q_T).
		\nonumber
	\end{equation}
	We finally need to consider the renormalized initial condtion \eqref{def1.1} of problem \eqref{beginequation}. Through the above analysis, we obtain that $\Theta_S^{\varepsilon} (u^{\varepsilon})$ is bounded in $L^{\infty}(Q_T)$, while $\frac{\partial \Theta_S^{\varepsilon}(u^{\varepsilon})}{\partial t}$ is bounded in $V^{\ast}+L^1(Q_T)$.  
	Using the embedding of lemma \eqref{lemma1}, we find that $\frac{\partial \Theta_S^{\varepsilon} (u^{\varepsilon})}{\partial t}$ is bounded in $L^{(p^+)'}(0,T;W^{-1,(p^+)'}(\Omega)$). As a consequence,an Aubin's type lemma(see corollary4 of \cite{Simon1986}) imply that 
	$\Theta_S^{\varepsilon} (u^{\varepsilon})$ lies in a compact set of $C([0,T];W^{-1,r}(\Omega))$, for any $r<\inf\left\{(p^+)',\frac{N}{N-1}\right\}$.
	It follows that $\Theta_S^{\varepsilon}(u^{\varepsilon})(t=0)=\Theta_S^{\varepsilon}(u_0^{\varepsilon})$ convergence to $\Theta_S(u)(t=0)$,strongly in $W^{-1,r}(\Omega)$. Furthermore, \eqref{bbqt} and Lebesgue dominated Convergence Theorem imply that 
	$\Theta_S^{\varepsilon}(u_0^{\varepsilon})\rightarrow \Theta_S(u_0)$, strongly in $L^1(\Omega)$. Therefore, $\Theta_S(u)(t=0)=\Theta_S(u_0)$, $a.e.$ in $\Omega$.
\end{proof}

\newpage
\bibliographystyle{plain}
\bibliography{ref.bib}
\end{document}